\documentclass[12pt]{article}
\usepackage{graphicx} 
\usepackage{color}
\usepackage{enumerate}
\usepackage{latexsym}
\usepackage[T1]{fontenc}
\usepackage{float}
\usepackage{amssymb,amscd,amsmath,amsfonts,amsthm}
\usepackage{multicol}
\usepackage{tabularx,environ,array}
\usepackage{tikz}
\usepackage{hyperref}
\usepgflibrary{shapes.geometric}

\usepackage{algorithm}
\usepackage{algpseudocode}
\usepackage{color}
\usepackage[draft]{changes}
\usepackage{enumerate}
\usepackage{graphicx}
\usepackage{multirow}
\usepackage{tikz}
\usepackage{tikz-qtree}
\usepackage{url}

\algnotext{EndFor}
\algnotext{EndIf}
\allowdisplaybreaks

\hypersetup{colorlinks=true, linkcolor=blue, citecolor=blue, urlcolor=blue}

\newtheorem{theorem}{Theorem}

\newtheorem{corollary}[theorem]{Corollary}

\newtheorem{definition}[theorem]{Definition}

\newtheorem{lemma}[theorem]{Lemma}

\newtheorem{proposition}[theorem]{Proposition}
\newtheorem{remark}[theorem]{Remark}

\newcommand{\adim}{\operatorname{adim}}
\newcommand{\adimms}{\operatorname{msad}}

\newcommand{\msl}{\{\hspace*{-0.1cm}|}
\newcommand{\msr}{|\hspace*{-0.1cm}\}}

\begin{document}

\title{On the $(k,\ell)$-multiset anonymity measure for social graphs}
\date{}
\author{Alejandro Estrada-Moreno$^1$, Elena Fern\'andez$^2$, Dorota Kuziak$^2$,\\  Manuel Mu\~noz-M\'arquez$^2$, Rolando Trujillo-Rasua$^1$, Ismael G. Yero$^3$\\
\\ 
$^1$Departament d'Enginyeria Inform\`atica i Matem\`atiques,\\ Universitat Rovira i Virgili, Campus Sescelades, Spain\\
\texttt\small{\{alejandro.estrada, rolando.trujillo\}@urv.cat}\\
$^2$Departamento de Estad\'istica e Investigaci\'on Operativa,\\ Universidad de C\'adiz, Spain \\
\texttt\small{\{elena.fernandez,dorota.kuziak,manuel.munoz\}@uca.es}\\
$^3$Departamento de Matem\'aticas,\\ Universidad de C\'adiz, Algeciras Campus, Spain\\
\texttt\small{ismael.gonzalez@uca.es}
}

\maketitle

\begin{abstract}
The publication of social graphs must be preceded by a rigorous analysis of privacy threats against social graph users.
When the threat comes from inside the social network itself, the threat is called an active attack, and the de-facto privacy measure used to quantify the resistance to such an attack is the $(k,\ell)$-anonymity.
The original formulation of $(k,\ell)$-anonymity represents the adversary's knowledge as a vector of distances to the set of attacker nodes. In this article, we argue that such adversary is too strong when it comes to counteracting active attacks.
We, instead, propose a new formulation where the adversary's knowledge is the multiset of distances to the set of attacker nodes.
The goal of this article is to study the
$(k,\ell)$-multiset anonymity from a graph theoretical point of view, while establishing its relationship to $(k,\ell)$-anonymity in one hand, and considering the $k$-multiset antiresolving sets as its theoretical frame, in a second one. That is, we prove properties of some graph families in relation to whether they contain a set of attacker nodes that breaks the  $(k,\ell)$-multiset anonymity.
From a practical point of view, we develop a linear programming formulation of the $k$-multiset antiresolving sets that allows us to calculate the resistance of social graphs against active attacks. This is useful for analysts who wish to know the level of privacy offered by a graph.
\end{abstract}\vspace{0.5mm}
\textbf{Math.\ Subj.\ Class. 2020:} 05C12, 05C76, 68R10, 68P27, 68M25\vspace{0.5mm}


\section{Introduction}

This article studies graph-theoretical properties of social graphs that impact their resistance to active privacy attacks (active attacks for short). The goal of an active attack is to learn sensitive information about users whose data was published in the form of an anonymous social graph, \textit{i.e.}, a graph where vertices represent supposedly anonymous users and edges their relations.
As opposed to passive attackers, active attackers can interact with other users prior the publication of the social network graph, making them significantly more threatening to an individual's privacy \cite{Backstrom2007}.

Active attackers re-identify a user in a (supposedly) anonymous graph as follows. Prior the publication of the social graph, a set of attacker nodes in the social network, known as \emph{attacker} nodes in the literature, influences the graph data by establishing connections with other victims and between themselves. Those connections are carefully crafted to make them unique for each victim.
Once the social network is published, the attackers use subgraph search to retrieve the set of attacker nodes. Victims are ultimately re-identified by looking for their connection patterns with respect to the set of  attacker nodes.

Defending against active attacks is notoriously difficult, as defenders have no means to detect attacker nodes reliably.
To calculate the re-identification probability of a victim, Trujillo-Rasua and Yero \cite{TRUJILLORASUA2016403} introduced $(k, \ell)$-anonymity.
Like most privacy models based on the popular notion of $k$-anonymity \cite{kanonymity}, $(k, \ell)$-anonymity partitions the set of users into anonymity sets and argues that users within the same anonymity set are indistinguishable. Because the adversary in an active attack is a set of attacker nodes $S = \{s_1, \ldots, s_l\}$ with connections to their victims in the social graph, the $(k, \ell)$-anonymity considers two users to be within the same anonymity set if their \emph{connection patterns} to $S$ are equal.
Trujillo-Rasua and Yero defined a connection pattern
to be the vector of (geodetic) distances from the set of attacker nodes to the victim node, leading to the following privacy measure.

\begin{definition}[$(k, \ell)$-anonymity]
Consider a vertex $u$ and a subset of vertices $S$ within a graph $G$.
\begin{itemize}
    \item The \emph{metric representation} of $u$ with respect to $S$ is
$r(u|S)=(d_G(u, s_1), \ldots, d_G(u, s_l))$ where $d_G(x,y)$ represents the {\em distance} between $x$ and $y$.
    \item The \emph{anonymity set} of a user $u$ with respect to $S$ is given by $S_{[u]} = \{v \in V(G) | r(u | S) = r(v |S)\}$.
    \item $S$ is called a \emph{$k$-antiresolving set} if $k$ is the largest integer such that
    for every vertex $v \in V(G) \setminus S$ it holds that $|S_{[u]}| \geq k$.
\end{itemize}
The graph $G$ is said to be \emph{$(k, \ell)$-anonymous} if every subset of vertices $S$ of $G$ with $|S| \leq \ell$ is a $k'$-antiresolving set with $k' \leq k$.
\end{definition}

In other words, a graph $G$ is $(k, \ell)$-anonymous, or satisfies $(k, \ell)$-anonymity, if for every possible set of attacker nodes $S$ and every possible victim $v \in V(G) \setminus S$, either $v$'s anonymity set is sufficiently large ($\geq k$) or the set of attacker nodes is impractically large ($> \ell$).


Since the introduction of $(k, \ell)$-anonymity, efforts have been made on understanding the properties of graphs that impact $k$ and $\ell$.
Some works have focused on determining the size of the smallest $k$-antiresolving sets in $G$, known as the \emph{$k$-metric antidimension} of $G$ \cite{Kratica2019KmetricAO,Tang2021,Trujillo2016}, others have studied the computational aspects of $(k, \ell)$-anonymity \cite{CHATTERJEE201953,DASGUPTA201987,Fernandez2023,TRUJILLORASUA2016403,Zhang2017}, and others have introduced graph-transformation techniques to obtain $(k, \ell)$-anonymous graphs with minimal perturbation \cite{Erfani2019,conditional,MauwRT22,MauwRT19,MauwTX16}.

A drawback of the $(k, \ell)$-anonymity privacy model is the assumption that the attacker subgraph $S$ can be retrieved entirely and that it contains no non-trivial automorphism. That is to say, $(k, \ell)$-anonymity assumes that the attacker can correctly re-identify all attacker nodes after anonymization.
Such assumption has been proven too strong \cite{MauwRT19} when the graph has been anonymised via perturbation. Indeed, existing perturbation techniques against active attacks, see \cite{Erfani2019,MauwRT18,conditional,MauwRT19,MauwTX16}, quickly diminish the probability of success of an active adversary retrieving the set of attacker nodes in the correct order.
That is why, in this article, we propose to drop the assumption that the attacker subgraph in the anonymised graph has no non-trivial automorphism. Fundamentally, this implies reinterpreting what is meant by a connection pattern from a victim to a set of attacker nodes. Instead of considering the connection pattern of a victim $v$ to be the vector of distances to $S$, which requires totally ordering the vertices in $S$, we propose to define it as the
\emph{multiset} of distances to $S$, thereby dropping the assumption that $S$ can be totally ordered by the attacker.
This leads to a new notion of $(k, \ell)$-anonymity, which we call $(k, \ell)$-multiset anonymity.

\begin{definition}[$(k, \ell)$-multiset anonymity]
Given a set of vertices $S$ and a vertex $v\in V(G)$, we define:
\begin{itemize}
    \item the \emph{multiset representation} of $v$ with respect to $S=\{s_1,\dots,s_l\}$ to be the multiset $m(v|S)=\msl d_G(v,s_1),\dots,d_G(v,s_r)\msr$;
    \item the \emph{anonymity set} of a user $u$ with respect to $S$ to be $\Tilde{S}_{[u]} = \{v \in V(G) | m(u | S) = m(v |S)\}$; and
    \item a $k$-\emph{multiset antiresolving set} ($k$-\emph{MARS} for short) to be a set $S\subset V(G)$ satisfying that $k$ is the largest integer such that for all $u\in V(G)\setminus S$ it holds that $|\Tilde{S}_{[u]}|\geq k$.

\end{itemize}
A graph is $(k, \ell)$-multiset anonymous $($or satisfies $(k, \ell)$-multiset anonymity$)$ if every subset of vertices $S$ with $|S| \leq \ell$ is a $k'$-MARS with $k' \leq k$.

\end{definition}

As an example for the previous concepts, consider the hypercube graph $Q_3$ drawn in Figure \ref{fig:Q_3}. Notice that any vertex of $Q_3$ is a $1$-MARS. Also, any two vertices at distance at most $2$ in $Q_3$ form a $2$-MARS, while two vertices at distance $3$ in $Q_3$ form a $6$-MARS. In addition, any other set of vertices of any cardinality always forms a $1$-MARS. These arguments allow to claim that $Q_3$ meets $(1,1)$-multiset anonymity, since $k=1$ is the smallest integer such that $Q_3$ contains a $1$-MARS of cardinality $1$. In practice this means that, there exist vertices in $Q_3$ that by themselves can uniquely re-identify other vertices in the graph. This represents the lowest privacy threshold a graph can offer.

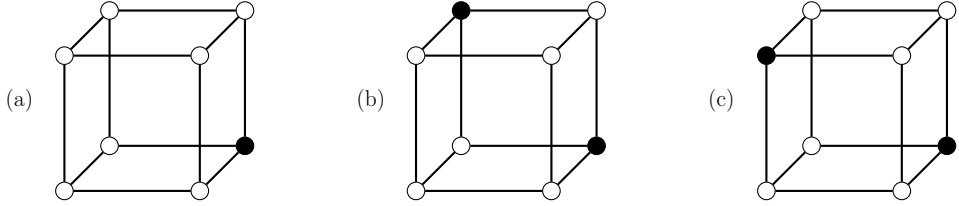
\begin{figure}[ht]
\centering
\begin{tikzpicture}[scale=0.6, transform shape]
\node [draw, shape=circle] (000) at  (0,0) {};
\node [draw, shape=circle] (001) at  (0,3) {};
\node [draw, shape=circle] (010) at  (3,0) {};
\node [draw, shape=circle] (011) at  (3,3) {};

\node [draw, shape=circle] (100) at  (1,1) {};
\node [draw, shape=circle] (101) at  (1,4) {};
\node [draw, shape=circle,fill=black] (110) at  (4,1) {};
\node [draw, shape=circle] (111) at  (4,4) {};

\draw[thick](000)--(001)--(011)--(010)--(000);
\draw[thick](100)--(101)--(111)--(110)--(100);
\draw[thick](000)--(100);
\draw[thick](001)--(101);
\draw[thick](011)--(111);
\draw[thick](010)--(110);
\draw (-1,2) node {\large(a)};
\end{tikzpicture}
\hspace{1cm}
\begin{tikzpicture}[scale=0.6, transform shape]
\node [draw, shape=circle] (000) at  (0,0) {};
\node [draw, shape=circle] (001) at  (0,3) {};
\node [draw, shape=circle] (010) at  (3,0) {};
\node [draw, shape=circle] (011) at  (3,3) {};

\node [draw, shape=circle] (100) at  (1,1) {};
\node [draw, shape=circle,fill=black] (101) at  (1,4) {};
\node [draw, shape=circle,fill=black] (110) at  (4,1) {};
\node [draw, shape=circle] (111) at  (4,4) {};

\draw[thick](000)--(001)--(011)--(010)--(000);
\draw[thick](100)--(101)--(111)--(110)--(100);
\draw[thick](000)--(100);
\draw[thick](001)--(101);
\draw[thick](011)--(111);
\draw[thick](010)--(110);
\draw (-1,2) node {\large(b)};
\end{tikzpicture}
\hspace{1cm}
\begin{tikzpicture}[scale=0.6, transform shape]
\node [draw, shape=circle] (000) at  (0,0) {};
\node [draw, shape=circle,fill=black] (001) at  (0,3) {};
\node [draw, shape=circle] (010) at  (3,0) {};
\node [draw, shape=circle] (011) at  (3,3) {};

\node [draw, shape=circle] (100) at  (1,1) {};
\node [draw, shape=circle] (101) at  (1,4) {};
\node [draw, shape=circle,fill=black] (110) at  (4,1) {};
\node [draw, shape=circle] (111) at  (4,4) {};

\draw[thick](000)--(001)--(011)--(010)--(000);
\draw[thick](100)--(101)--(111)--(110)--(100);
\draw[thick](000)--(100);
\draw[thick](001)--(101);
\draw[thick](011)--(111);
\draw[thick](010)--(110);
\draw (-1,2) node {\large(c)};
\end{tikzpicture}
\caption{Bold vertices form a $1$-MARS in (a); a $2$-MARS in (b); and a $6$-MARS in (c).}\label{fig:Q_3}
\end{figure}

\noindent \emph{Contributions.} Given the obvious evolution from $k$-antiresolving sets to $k$-MARS, our first contribution is the study of their relation. Concretely, we study  how existing results on the $k$-metric antidimension stands in the multiset version. We study the largest value of $k$, denoted $\kappa(G)$, such that there exists a $k$-multiset antiresolving in a graph $G$. This parameter is relevant, as it bounds the maximum level of privacy that a graph can offer, when there is no limit on the number of attacker nodes.
We study the $k$-multiset antidimension of various graphs families, such as trees, wheels, and bipartite graphs. The $k$-multiset antidimension gives the minimum $\ell$, i.e. the minimum cost of the attack in terms of number of attacker nodes, for a graph to be $(k, \ell)$-multiset anonymous.  Lastly, we provide an integer linear programming formulation to determine the privacy of a social graph measured by means of the $(k, \ell)$-multiset anonymity privacy model. We implemented our formulation, and tested it on benchmark instances on general graphs, both sparse and dense, with up to 100 vertices which were already used in \cite{Fernandez2023}, and with two cycle instances with 37 and 40 vertices, respectively. The obtained results support our theoretical findings and highlight the difficulty in finding proven optimal solutions, except for instances where any smallest $k$-MARS has cardinality one, especially for sparse graphs. The effect of the parameter $k$ on the difficulty of solving the instances seems to depend on the density of the input graph: sparse instances tend to be less demanding for smaller values of $k$, whereas dense graphs show the opposite behavior.

\section{Related work and preliminaries}

In this section, we review in chronological order
existing resolvability notions that preceded $k$-multiset antiresolving sets. This helps to further place our work in context. We also introduce necessary notation and terminology for the remaining sections.
Lastly, we provide a preliminary analysis on the relation between $k$-antiresolving sets and $k$-multiset antiresolving sets, which serves as a warm-up for the theoretical results that follow.

\subsection{The path towards $k$-multiset anonymity}


The $k$-multiset antiresolving sets are related to various resolvability notions on graphs, such as the metric dimension, the multiset-metric dimension and the $k$-metric antidimension. Next, we review these graphs parameters in chronological order.

\noindent \textbf{The metric dimension:}  A given set of vertices $S\subset V(G)$ is a \textit{resolving set} for a connected graph $G$ if any two vertices $x,y\in V(G)$ satisfy that $r(x|S)\ne r(y|S)$. This is traduced into the fact that, for such vertices $x,y$ there exists a vertex $s\in S$ such $d_G(x,s)\ne d_G(y,s)$. A resolving set of the smallest possible cardinality is called a \textit{metric basis} for $G$, and the cardinality of a metric basis represents the \textit{metric dimension} of $G$, denoted by $\dim(G)$. These notions were first (and independently) presented in \cite{Harary1976,Slater1975}, and they are very well studied nowadays.

Among the main results, it is known that computing the metric dimension of a graph is NP-hard, even when restricted to some specific graph classes like for instance planar graphs (see \cite{diaz2017} for the planar case). There are also several contributions on finding the metric dimension of a huge number of graph classes, as well as, a large number of possible applications of this topic in some related areas. In order to not give a large number of references, to see more information on this parameter, applications and combinatorial or computational results, we suggest the two recent surveys \cite{kuziak2021,tillquist2023}.

\noindent \textbf{The multiset dimension:} The metric representation of a vertex $v$ with respect to a set of vertices $S$ has traditionally assumed an order on $S$. That assumption was challenged in 2017 by Simanjuntak, Vetrík, and Mulia, who introduced the notion of multiset representation \cite{simanjuntak2019multiset} by looking at the multiset of distances rather than at the vector of distances.
They observed that the multiset representation with respect to $S$, unlike the standard metric representation, is not unique for each vertex in $S$.
This suggests that a multiset resolving set may not always exist, which is indeed the case for the complete graph. Gil-Pons et. al.  \cite{GILPONS2019124612} later proposed an alternative formulation, called the outer multiset dimension, where vertices internal to $S$ are excluded from the distinguishability requirement. They proved that computing the outer multiset dimension is NP-hard and studied the outer multiset dimension for some graph families. Some other recent results on the (outer) multiset dimension of graphs are for instance \cite{hakanen2024,klavzar-outermult}.

\noindent \textbf{The $k$-metric antidimension:}  Given a graph $G$ and an ordered set of vertices $S=\{v_1,\dots,v_r\}\subset V(G)$, the \textit{metric representation} of a vertex $x\in V(G)$ with respect to $S$ is the vector $r(x|S)=(d_G(x,v_1),\dots,d_G(x,v_r))$. A set $S\subset V(G)$ is a $k$-\textit{antiresolving set} of $G$ if $k$ is the largest integer such that for every vertex $v \in V(G)\setminus S$ there exist other $k-1$ vertices having the same metric representation with respect to $S$ as the vertex $v$. Also, the $k$-\textit{metric antidimension} of $G$ represents the cardinality of a smallest $k$-antiresolving set in $G$, and it is denoted by $\adim_k(G)$. Moreover, by $\kappa'(G)$ we represent the largest integer $k$ such that $G$ contains a $k$-antiresolving set, and say that $G$ is $\kappa'(G)$-\textit{metric antidimensional}. These concepts were first presented in \cite{TRUJILLORASUA2016403}.

The earliest preliminary results on the $k$-metric antidimension were given in \cite{TRUJILLORASUA2016403}, particularly on paths, cycles, complete bipartite graphs and trees. The same authors characterized in \cite{Trujillo2016} the trees and unicyclic graphs that do not contain a $k$-antiresolving set with $k > 1$. Such property is considered a worst-case scenario as far as privacy is concerned. Jozef et. al. \cite{Kratica2019KmetricAO}  and Tang et. al. \cite{Tang2021}
have also studied the $k$-metric antidimension. The former focused on some generalized Petersen graphs; the latter on four families of wheel-related social graphs, namely, Jahangir graphs, helm graphs, flower graphs, and sunflower graphs.

The computational complexity of the $k$-metric antidimension problem (and its variants) has been studied independently in \cite{Zhang2017} and
\cite{CHATTERJEE201953}, both proving that computing the $k$-metric antidimension is NP-hard. Consequently, a few heuristics have been   proposed to calculate the $k$-metric antidimension. For example, DasGupta et. al. \cite{DASGUPTA201987} studied the anonymity of eight real-life social graphs, and various synthetic graphs, by looking into the smallest $k$-antiresolving sets. Fern\'andez et. al. \cite{Fernandez2023} calculated via integer programming formulations the $k$-metric antidimension of  cylinders, toruses, 2-dimensional Hamming graphs, and random graphs. Their results indicate that only the 2-dimensional Hamming graphs and some general random dense graphs offer good privacy properties. Trujillo-Rasua and Yero in \cite{TRUJILLORASUA2016403} introduced a true-biased algorithm for determining whether a graph is $(k, \ell)$-anonymous. Their algorithm has high accuracy, but it is limited to graphs of small size.

\noindent \textbf{The $k$-multiset antidimension:}
The type of metric dimension related parameter we study in this article is the $k$-\emph{multiset antidimension}, denoted $\adimms_k(G)$, and defined as the cardinality of a smallest $k$-MARS of $G$.
In addition, a $k$-\emph{multiset antiresolving basis} ($k$-\emph{MARB}) is a $k$-MARS of cardinality $\adimms_k(G)$.

It is natural to think that there are no $k$-MARS for any integer $k$ in a graph $G$. In concordance with this, from now on, we say that $\kappa(G)$ is the largest integer $k$ such that $G$ contains a $k$-MARS, and shall also say that $G$ is a $\kappa(G)$-\textit{multiset antidimensional graph}. In consequence, before studying the $k$-multiset antidimension of a graph $G$ (for suitable values of $k$), one needs to compute the value $\kappa(G)$ for such $G$. We may observe also that, even in the case that a given graph $G$ will be $\kappa(G)$-multiset antidimensional graph, which means that $G$ contains a $\kappa(G)$-MARS, this does not mean that for every integer $k\in  \{1,\dots,\kappa(G)\}$ the graph $G$ contains a $k$-MARS. Indeed, as will show in our exposition, there are graphs $G$ and values $k'$ with $1<k'<\kappa(G)$ such that $G$ does not contain $k'$-MARS. In such situation, we shall assume that $\adimms_{k'}(G)=\infty$.

\subsection{Notation and terminology}

For an integer $n\ge 1$, we shall write $[n]=\{1,\dots, n\}$. The \textit{eccentricity} of a vertex $v\in V(G)$ is the largest possible distance between $v$ and any other vertex $x\in V(G)$. The \textit{diameter} $D(G)$ of $G$ is the largest possible eccentricity of a vertex of $G$, while the \textit{radius} $r(G)$ is the smallest possible eccentricity of any vertex of $G$. Notice that the \textit{center} of a graph is the set of vertices having eccentricity equal to the radius of the graph. For a vertex $v\in V(G)$, by $N_G(v)$ we represent the set of neighbors of $v$.

In addition, the following perspective on the $k$-MARS is very useful. Given a graph $G$ and a set $S\subset V(G)$, we define the following equivalence relation $\mathcal{R}_S$. Two vertices $x,y\in V(G)\setminus S$ are related by $\mathcal{R}_S$ if $m(x|S)=m(y|S)$ ($x,y$ have the same multiset representation with respect to $S$).
Now, for a given set $S\subset V(G)$, we consider $\mathcal{Z}_S=\{Z^1,\dots,Z^r\}$, for some $r\ge 1$, as the set of equivalence classes defined by $\mathcal{R}_S$.
With this terminology in mind, it is readily observed that any vertex set $S\subseteq V(G)$ is a $k$-MARS of $G$ with $k=\min\{|Z^i|\,:\,Z^i\in \mathcal{Z}_S\}$.

The terminology above shall be used throughout our whole exposition while presenting some arguments in the proofs.

\subsection{On the relation between $k$-antiresolving sets and $k$-MARS}

If we consider the case $k=1$, then there is no difference between the notions of vectors and multisets with respect to a set of cardinality one. This means that $1$-MARS, and $1$-multiset antidimension are exactly the same as $1$-antiresolving sets, and $1$-metric antidimension, respectively. Thus, in our investigation, we are mainly focused on the cases $k\ge 2$.

An initial basic result that we observe relates $\adimms_k(G)$ and $\adim_k(G)$. This is based on the following fact. If two vertices have the same metric representation, then they have the same multiset representation.

\begin{remark}\label{rem:rm-bounds}
If $S$ is $k$-antiresolving set of a graph $G$, then $S$ is a $k'$-MARS for some $k'\ge k$.
\end{remark}

\begin{corollary}\label{cor:adim-adimms}
If $\adim_k(G)=t$ for some positive integer $k$, then there is an integer $k'\ge k$ for which $\adimms_{k'}(G)\le t$.
\end{corollary}

For instance, by taking $k=4$ and $k'=8$ in the graph $G^*$ of Figure \ref{figure-G-*}, it is satisfied that $\adim_4(G^*)=2$ and $\adimms_{8}(G^*)=2$.

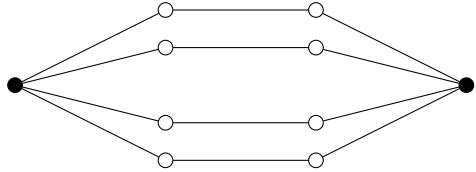
\begin{figure}[ht]
\centering
\begin{tikzpicture}[scale=.5, transform shape]
\node [draw, shape=circle, fill=black] (a) at  (-6,0) {};
\node [draw, shape=circle, fill=black] (c) at  (6,0) {};

\node [draw, shape=circle] (b1) at  (-2,-2) {};
\node [draw, shape=circle] (b2) at  (-2,-1) {};
\node [draw, shape=circle] (b3) at  (-2,1) {};
\node [draw, shape=circle] (b4) at  (-2,2) {};

\node [draw, shape=circle] (b11) at  (2,-2) {};
\node [draw, shape=circle] (b21) at  (2,-1) {};
\node [draw, shape=circle] (b31) at  (2,1) {};
\node [draw, shape=circle] (b41) at  (2,2) {};

\draw(b1)--(a)--(b2);
\draw(b3)--(a)--(b4);
\draw(b11)--(c)--(b21);
\draw(b31)--(c)--(b41);
\draw(b1)--(b11);
\draw(b2)--(b21);
\draw(b3)--(b31);
\draw(b4)--(b41);

\end{tikzpicture}
\caption{The graph $G^*$ with $\adim_4(G^*)=2$ and $\adimms_{8}(G^*)=2$.}
\label{figure-G-*}
\end{figure}

On the other hand, it was noted in \cite{TRUJILLORASUA2016403}, that for any graph $G$ of maximum degree $\Delta$ it holds that $\kappa'(G)\le \Delta$. One could think that an analogous situation could occur for $\kappa(G)$. However, this is far from the reality. Consider for instance the graph $G^*$ given in Figure \ref{figure-G-*}. We can easily note that $G^*$ is $\Delta$-metric antidimensional ($\kappa'(G^*)=\Delta$), where the bold vertices form a standard $\Delta$-antiresolving set, while the same set of bold vertices forms a $(2\Delta)$-MARS, \emph{i.e.}, $\kappa(G^*)\ge 2\Delta$ (in this specific case, we note that indeed $\kappa(G^*)=2\Delta$).

Moreover, we might notice that there is in general no monotony in the value of $\adimms_{k}(G)$ with respect to $k$. Namely, it cannot be claimed as a general rule, that for any suitable $k$, $\adimms_{k-1}(G)\le \adimms_{k}(G)$ or $\adimms_{k-1}(G)\ge \adimms_{k}(G)$. It is even possible to see that, if a graph $G$ is $\kappa(G)$-multiset antidimensional, then, as already mentioned, there are not necessarily $k$-multiset antiresolving sets for some $1\le k\le \kappa(G)$.

\section{Combinatorial properties of the $k$-multiset antidimension}

Since the $k$-multiset antidimension is the theoretical basis of the $(k,\ell)$-multiset anonymity measure, it is then natural that mathematical properties of it are desirable. We first consider the problem of finding $\kappa(G)$ for a graph $G$.

\subsection{Finding the value $\kappa(G)$}

Straightforward bounds for $\kappa(G)$ for a graph of order $n$ are the following ones.

\begin{proposition}
\label{prop:trivial-bounds}
If $G$ is a graph of order $n$, then $1\le \kappa(G)\le n-1$. Moreover,
$\kappa(G)= n-1$ if and only if $G$ has a vertex of degree $n-1$.
\end{proposition}

\begin{proof}
The bounds are straightforward. Now, assume $G$ is $(n-1)$-multiset antidimensional. Thus, there is a set $S$ which is an $(n-1)$-MARS, which means $|S|=1$ and that every vertex not in $S$ has the same distance to the unique vertex of $S$. Clearly, this can be only possible whether such vertex in $S$ is adjacent to any other vertex not in $S$. If $G$ has a vertex of degree $n-1$, then such vertex is clearly a $(n-1)$-MARS, which completes the proof of the second statement.
\qed\end{proof}

With respect to the lower bound above, if $G$ is the unique connected graph of order $n=2$, then it is clear that $\kappa(G)=1$, and thus it is $1$-multiset antidimensional. Now, if we consider any graph $G$ with at least three vertices, then it might be true that $\kappa(G)\ge 2$. However, we have not been able to prove this fact in general, but we support this statement with several non-trivial graphs classes that are next presented.

\begin{theorem}
\label{th:trees-kappa}
For any tree $T$ with at least three vertices, $\kappa(T)\ge 2$.
\end{theorem}

\begin{proof}
We consider the center of the tree $T$, i.e., the set of vertices of $T$ having the eccentricity equal to the radius of $T$. Notice that the center of any tree is either a vertex or two adjacent vertices.

Assume first that $T$ the center of $T$ is formed by a unique vertex $r$. Let $D_i(r)$ be the set of vertices at distance $i$ from $r$ with $i\in [r(G)]$. Since $r$ represents the center of $T$, there are at least two vertices in each $D_i(r)$ for every $i\in [r(G)]$. Thus, it can be readily observed that the set $S=\{r\}$ is a $k'$-MARS with $k'=\min\{|D_i(r)|\,:\,i\in [r(G)]\}$. Since  $|D_i(r)|\ge 2$ for every $i\in [r(G)]$, we deduce that $\kappa(T)\ge 2$.

Assume next that $T$ has a center formed by two adjacent vertices $r,s$. This means that the set of vertices of $T$ can be partitioned into disjoint sets $X_i(r)\subsetneq D_i(r)$ and $X_i(s)\subsetneq D_i(s)$ with $i\in [r(G)]$ such that if $x\in X_j(r)$, then $d_T(x,s)=j+1=d_T(x,r)+1$, and if $x\in X_j(s)$, then $d_T(x,r)=j+1=d_T(x,s)+1$.

Note that each of the sets $X_i(r)$ and $X_i(s)$ with $i\in [r(G)]$ is not empty, based on the fact that $r,s$ form the center of $T$. Let $S=\{r,s\}$. Based on the definitions of the sets $X_i(r)$ and $X_i(s)$, if $x\in X_j(r)$, then $m(x|S)=\msl j,j+1\msr$, while if $x'\in X_j(s)$, then $m(x'|S)=\msl j,j+1\msr$. Thus, both vertices $x,x'$ have the same multiset representation with respect to $S$, and so, $S$ is a $k'$-MARS with $k'=\min\{|X_i(r)|+|X_i(s)|\,:\,i\in [r(G)]\}$. Since $|X_i(r)|\ge 1$ and $|X_i(s)|\ge 1$ for every $i\in [r(G)]$, we deduce that $\kappa(T)\ge 2$ as well.
\qed\end{proof}

The result above arises a question concerning characterizing the class of trees $T$ satisfying that $\kappa(T)=2$. We now continue with other basic families of graphs, which in addition, supports our statement about the lower bound for $\kappa(G)$.

\begin{proposition}\label{graph-multi-antidim}
The following statements holds for any integers $r,t$, with $r\ge t$, and $n\ge 2$.
\begin{itemize}
  \item[{\rm (i)}] If $t=1$, then $\kappa(K_{r,t})=r+t-1$, and otherwise $\kappa(K_{r,t})=r+t-2$.
  \item[{\rm (ii)}] $\kappa(P_n)=2$.
\end{itemize}
\end{proposition}

\begin{proof}
(i) If $t=1$, then clearly $K_{r,1}$ is a star graph with $r$ leaves that has a vertex of degree equal to the order minus one. Thus, the first conclusion follows from Proposition \ref{prop:trivial-bounds}. If $t>1$, then a set of two vertices formed by one vertex of each bipartition set of $K_{r,t}$ is clearly an $(r+t-2)$-MARS, and there are not $(r+t-1)$-antiresolving sets by using again Proposition \ref{prop:trivial-bounds}. Thus, the second conclusion also follows.

(ii) We observe that, if $n$ is odd, then the vertex of $P_n$ equidistant from both leaves of $P_n$ is a $2$-MARS of $P_n$. Also, if $n$ is even, then the two leaves of $P_n$ form a $2$-MARS as well. Suppose now $S$ is a $k$-MARS of $P_n$ for some $k\ge 3$. Assume $P_n=v_1v_2\dots v_n$ and consider the following situations.
\begin{itemize}
  \item If one leaf, say $v_1$, does not belong to $S$, then any other vertex which has the same multiset as $v_1$ must be a vertex $v_i$ such that $i>j$, for every $v_j\in S$. Hence, we readily note that there could be only one vertex more having the same multiset representation as $v_i$. This means that $S$ is a $2$-MARS, a contradiction.
  \item If both leaves of $P_n$ belong to $S$, then consider the smaller $i$ such that $v_i\notin S$. Now, we observe that any other vertex which has the same multiset as $v_i$ must be a vertex $v_j$ for which $d(v_j,v_n)=d(v_i,v_1)$ and this is a unique vertex. Thus, again $S$ can only be a $2$-MARS, a contradiction.
\end{itemize}
As a consequence of the cases above, we get that $\kappa(P_n)=2$ as we need.
\qed\end{proof}

\subsection{Computing the $k$-multiset antidimension of graphs}\label{subsec:computing}

We center our attention in this section into finding the value of $\adimms_k(G)$ for several classes of graphs $G$. We begin with the case of complete bipartite graphs.

\begin{proposition}
\label{prop:complete-bip}
Let $r,t$ be two positive integers with $r\ge t$.
\begin{itemize}
\item[{\rm (i)}] If $1< k< t$, then $\adimms_k(K_{r,t})=t-k$.
\item[{\rm (ii)}] If $k=t$, then $\adimms_k(K_{r,t})=\left\{\begin{array}{ll}
                                         1, & \mbox{if $r>t$}, \\
                                         r, & \mbox{if $r=t$}.
                                       \end{array}\right.
$
\item[{\rm (iii)}] If $t< k\le r$, then $\adimms_k(K_{r,t})=r+t-k$.
\item[{\rm (iv)}] If $($$r< k\le r+t-2$ is even and $r+t$ is even$)$ or $($$r< k\le r+t-2$ is odd and $r+t$ is odd$)$, then $\adimms_k(K_{r,t})=r+t-k$.
\item[{\rm (v)}] If $($$r< k\le r+t-2$ is even and $r+t$ is odd$)$ or $($$r< k\le r+t-2$ is odd and $r+t$ is even$)$, then $K_{r,t}$ does not contain any $k$-MARS.
\end{itemize}
\end{proposition}

\begin{proof}
From Proposition \ref{graph-multi-antidim} (i) we know that $\kappa(K_{r,t})=r+t-1$ if $t=1$, and that $\kappa(K_{r,t})=r+t-2$ otherwise. Let $U$ and $V$ be the two bipartition sets of $K_{r,t}$
with $|U|=r$ and $|V|=t$. We first assume that $t< k\le r$. Let $A\subseteq U$ with $|A|=k$ and let be $S=(V\cup U)-A$. Notice that if $k=r$, then $A=U$ and so, $S=V$. Since any vertex $v\notin S$ (or equivalently $v\in A$) is adjacent to every vertex of $V$ and it has distance two to every vertex in $U-A$, we have that all the vertices of $A$ have the same multiset representation with respect to $S$. As $|A|=k$, it follows that $S$ is a $k$-MARS and $\adimms_k(K_{r,t})\le r+t-k$. Now, suppose $\adimms_k(K_{r,t})<r+t-k$ and let $S'$ be a $k$-MARS of $K_{r,t}$.

If there exist more than $k$ vertices of $U$ not in $S'$, then for any
vertex $u\in U-S'$ there exist at least $k$ vertices not in $S'$ which,
together with $u$, have the same multiset representation with respect to $S'$. So, $S'$ is not a $k$-MARS, but a $k'$-MARS for some
$k'\ge k+1$, which is a contradiction. Thus, $|U-S'|\le k$. Now, if there exists at least one vertex $v\in V$ which is not in $S'$, then such a vertex must have the same multiset representation with respect to $S'$ as those vertices in $U-S'$. Moreover, it must happen that $|S'\cap U|=|S'\cap V|$, otherwise the vertices in $(U\cup V)-S'$ will not have the same multiset representation. Thus, $k=|(U\cup V)-S'|=r+t-|S'|< r+t-(r+t-k)=k$, which is also not possible.
Therefore, we obtain that $\adimms_k(K_{r,t})=r+t-k$.

Now assume $k=t$. If $r=t$, then clearly $U$, $V$ and half of vertices of $U$ together with half of vertices of $V$ (whether $r$ is an even number) are the only possibilities for a $k$-MARS of $K_{r,t}$. Thus, $\adimms_k(K_{r,t})=r$. On the contrary, if $r>t$, then for any vertex $z\in V$ it holds that the $t=k$ vertices in $U$ have the same multiset representation with respect to $\{z\}$ and for any vertex $w\in V-\{z\}$ there are at least $k-1$ vertices having the same multiset representation as $w$ with respect to $\{z\}$. Thus, $\{z\}$ is a $k$-MARS of $K_{r,t}$ of minimum cardinality, or equivalently, $\adimms_k(K_{r,t})=1$.

We now consider that $1< k\le t$. Let $Y\subseteq V$ with $|Y|=k$ and let $Q=V-Y$. Hence, for any vertex $v\in Y$, there exist exactly $k-1$ vertices, such that all of them, together with $v$, have the same
multiset representation with respect to $Q$. Moreover, for any vertex $u\in U$, there exist at least $k$ vertices having the same
multiset representation as $u$ with respect to $Q$. Thus, $Q$ is a $k$-MARS and $\adimms_k(K_{r,t})\le t-k$.

Now, suppose that $\adimms_k(G)<t-k$ and let $Q'$ be a $k$-MARS in $K_{r,t}$. Hence, there exist more than $k$ vertices of $U$ not in
$Q'$ or there exist more than $k$ vertices of $V$ not in $Q'$. Thus, in any
of both possibilities we obtain that $Q'$ is not $k$-MARS, but a $k'$-MARS for some $k'\ge k+1$, a contradiction. As a consequence, $\adimms_k(K_{r,t})= t-k$.

Finally, assume $r< k\le r+t-1$. Let $A$ be a $k$-MARB for $K_{r,t}$. Since $k>r\ge t$, it must happen $U-A\ne \emptyset$ and $V-A\ne \emptyset$. Thus, in order that vertices in $U-A$ and in $V-A$ will have the same multiset representation, it must happen that $|A\cap U|=|A\cap V|$. This means that $A$ has even cardinality and that all the vertices outside of $A$ have the same multiset representation w.r.t. $A$. Thus, if $k$ and $r+t$ have different parity, then it is not possible for $A$ to be a $k$-MARS. Thus, we may consider $k$ and $r+t$ have the same parity. If $|A|<r+t-k$, then $A$ is a $k'$-MARS for some $k'>k$, which is not possible. Consequently, $\adimms_k(K_{r,t})= |A|\ge r+t-k$. On the other hand, it can be checked that any set of cardinality $r+t-k$ having the same number of vertices in each bipartition set of $K_{r,t}$ is a $k$-MARS, which completes the proof.
\qed\end{proof}

Since any tree $T$ satisfies that $2\le \kappa(T)\le |V(T)|-1$, and indeed both bounds are achieved (paths $P_n$ satisfy that $\kappa(T)=2$ and stars $K_{1,t}$ satisfy that $\kappa(K_{1,t})=t$), it seems that finding the exact value of $\adimms_k(T)$ for any tree $T$ is a challenging problem. We next give some partial results on this direction, and begin with the straightforward case of paths.

\begin{remark}
For any integer $n\ge 2$, $\adimms_1(P_n)=1$ and $\adimms_2(P_n)=\left\{\begin{array}{ll}
                             1, & \mbox{if $n$ is odd,} \\
                             2, & \mbox{if $n$ is even.}
                             \end{array}
  \right.$
\end{remark}

On the other hand, note that stars are indeed complete bipartite graphs having one bipartition set of cardinality one, and such graphs were already studied in Proposition \ref{prop:complete-bip}. From now on we consider trees that are neither paths nor stars. This means, among other things, that the trees in study have maximum degree and diameter at least three.

We observe that any leaf of a tree $T$ is a $1$-MARS and so, $\adimms_1(T)=1$. In this sense, we next center the attention into the case $k\ge 2$, and precisely analyze first when $k=2$.

\begin{lemma}
\label{lem:two-vertices-tree}
Let $T$ be a tree of diameter at least three. Then, any two not adjacent vertices $x,y\in V(T)$ at odd distance form a $k$-MARS for some $k\le 2$.
\end{lemma}

\begin{proof}
Let $S=\{x,y\}$. Since $d_T(x,y)\ge 3$ ($x,y$ are not adjacent) is odd, the shortest $x,y$-path $P=xw_1w_2\cdots w_ry$ (with $r\ge 2$) contains an even number of vertices, i.e., $r$ is even. It can be readily seen that $\{w_1,w_r\}, \{w_2,w_{r-1}\},\dots,\{w_{r/2},w_{r/2+1}\}$ form equivalence classes in the equivalence relation $\mathcal{R}_S$.

Since all of these equivalence classes have cardinality two, we deduce that $S$ can be a $k$-MARS only if $k\le 2$.
\qed\end{proof}

From the result above, if any other equivalence class from the equivalence relation $\mathcal{R}_S$ has cardinality at least two, then clearly such $S$ is a $2$-MARS. If on the contrary, there is at least one equivalence class of cardinality one, then $S$ is a $1$-MARS. This allows to conclude the following result.

\begin{corollary}
For any tree $T$ of diameter at least three, $1\le \adimms_2(T)\le 2$.
\end{corollary}

The result above immediately suggest characterizing the trees $T$ for which either $\adimms_2(T)=1$ or $\adimms_2(T)=2$. We next focus on this fact. To this end, we need the following terminology and notation. Given a vertex $x\in V(T)$ and an integer $i\in [\epsilon(x)]$, by $D_i(x)$ we represent the set of vertices at distance $i$ from $x$.

\begin{proposition}
\label{prop:trees-adimms-1}
Let $T$ be a tree. Then $\adimms_2(T)=1$ if and only if there is a vertex $x\in V(T)$ and an integer $i\in [\epsilon(x)]$ such that $|D_i(x)|=2$.
\end{proposition}

\begin{proof}
Clearly, if $x$ satisfies that $|D_i(x)|=2$ for some integer $i\in [\epsilon(x)]$, then $\{x\}$ is a $2$-MARS, and so, $\adimms_2(T)=1$. On the other hand, if $\adimms_2(T)=1$, then there must be a vertex $x'$ such that the equivalence relation $\mathcal{R}_{\{x'\}}$ has an equivalence class of cardinality two. This happens only in the case that there will be a set $D_j(x')$ for some $j\in [\epsilon(x')]$ with $|D_j(x')|=2$. Thus, the proof is completed.
\qed\end{proof}

As a consequence of the result above, we deduce that a tree $T$ satisfies $\adimms_2(T)=2$ if and only if every vertex $x\in V(T)$ and every integer $i\in [\epsilon(x)]$ hold that $|D_i(x)|\ne 2$. Also, note that the result from Proposition \ref{prop:trees-adimms-1} could be applied to design an algorithm that can check in polynomial time whether a given tree satisfies $\adimms_2(T)=1$.

Since the situations $k=1$ and $k=2$ are already settled for trees, we next continue with the cases $k\ge 3$. In this situation, by using similar arguments as the ones used in Lemma \ref{lem:two-vertices-tree}, it can be readily observed that if a $k$-MARS of a tree $T$ contains two components, then $k\le 2$. Thus, the following result is deduced, and its proof omitted.

\begin{lemma}
\label{lem:k-3-connected}
If $S$ is a $k$-MARS of a tree $T$ for some $k\ge 3$, then $S$ induces a connected subtree of $T$.
\end{lemma}

Since finding the $k$-multiset antidimension of a tree when $k\ge 3$ seems to be a hardworking task, we now concentrate on the case of complete binary trees. That is, given an integer $d\ge 1$, a tree $T_{d}$ is a \textit{complete binary tree}, if it has a root vertex $r$ of degree $2$ such that all the leaves of $T_{d}$ are at distance exactly $d$ from $r$, and any other vertex of $T_{d}$ has degree $3$. This means that every non leaf vertex of $T_{d}$ has $2$ children. Clearly, if $d=1$, then $T_{d}$ is simply a the path $P_3$. Since paths have already been considered, we assume $d\ge 2$. The value $d$ is frequently called as the \textit{height} of $T_d$. Complete binary trees are indeed very well known and have been studied in several investigations. We next give some partial results on their $k$-multiset antidimension (although already considered, the result includes also the case $k=2$).

\begin{proposition}
\label{pro:binary}
For any integer $d\ge 2$, and any $k\in \{2,2^2,\dots 2^d\}$ it holds,
$$\adimms_k(T_{d})\le \sum_{i=1}^{\log_2 k} d^{i-1}.$$
\end{proposition}

\begin{proof}
Let $S$ be the set of vertices of $T_{d}$ containing the root $r$ and all the vertices at distance at most $\log_2 k$ from $r$ (for instance, if $k=4=2^2$, then $S$ is formed by $r$ and the two children of it). It can be readily observed that each equivalence class of the equivalence relation $\mathcal{R}_{S}$ is precisely formed by the vertices having the same distance to the root $r$, which are not in $S$ (for example, again if $k=2^2$, then the vertices at distance $3$, $4$, $\dots$, $d$ from $r$, form each of the equivalence classes of $\mathcal{R}_{S}$). Now, observe that these equivalence classes have cardinality $2^{k}$, $2^{k+1}, \dots, 2^d$. Thus, $S$ is a $2^k$-MARS, and the upper bound follows since $|S|=\sum_{i=1}^{\log_2 k} d^{i-1}$.
\qed\end{proof}

If $k=2$ in the result above, then we obtain that $\adimms_k(T_{d})\le 1$ for any integer $d\ge 2$, which is the exact value in this case. However, we remark that the bound above is not reached in general. To observe this, see for instance the binary tree $T_4$. In such a situation, the bound from Proposition \ref{pro:binary} gives $\adimms_4(T_{4})\le 3$, while indeed $\adimms_4(T_{4})=2$ (see Figure \ref{fig:T_4} for an example of a $4$-MARS of cardinality $2$).

The case $k=3$ can be also considered for the case of binary trees, and the following result is then obtained.

\begin{proposition}
For any integer $d\ge 2$, $\adimms_3(T_{d})=\left\{\begin{array}{ll}
    \infty; & \mbox{ if $d=2$}, \\
    2; & \mbox{ if $d=3$}, \\
    1; & \mbox{ if $d>3$}.
\end{array}\right.$
\end{proposition}

\begin{proof}
First, some simple computations allow to check that $T_2$ has no $3$-MARS, and so the conclusion follows in this case. Let now $d=3$. Since $T_3$ has only a few vertices, one can check that no single vertex form a $3$-MARS. Thus $\adimms_3(T_{d})\ge 2$. We consider a set $S$ formed by the root $r$ and exactly one of the children of $r$, say $r_1$. Now, the two children of $r_1$ together with the children of $r$, other than $r_1$, form an equivalence class of $\mathcal{R}_{S}$ (with cardinality $3$) having multiset representation $\msl 1, 2 \msr$. Since $d=3$, there are a few other equivalence classes in $\mathcal{R}_{S}$ of cardinality larger than $3$ having multiset representations $\msl x,x+1 \msr$ with $x\in \{2,3\}$. Therefore, $S$ is a $3$-MARS, and the equality follows in such case.

Finally, assume next $d\ge 4$. We consider now $S$ is a set formed by exactly one of the children of $r$, say $r_1$. Now, the two children of $r_1$ together with the root $r$, form an equivalence class of $\mathcal{R}_{S}$ (with cardinality $3$) having multiset representation $\msl 1 \msr$. Since $d\ge 4$, there are several other equivalence classes in $\mathcal{R}_{S}$ of cardinality larger than $3$ having multiset representations $\msl x \msr$ with $x\in \{2,\dots,d\}$. Therefore, $S$ is a $3$-MARS, and the required equality is obtained.
\qed\end{proof}

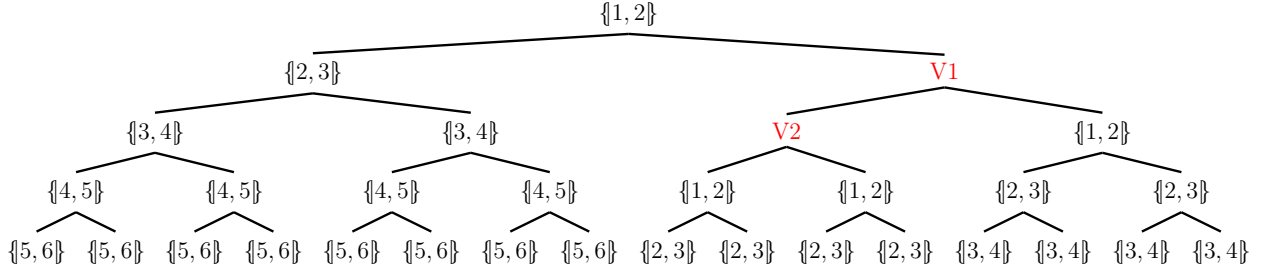
\begin{figure}[ht]
\centering
\begin{tikzpicture}[scale=0.75]
\tikzset{edge from parent/.append style={very thick}}
\Tree [.{$\msl 1,2 \msr$}
		[.{$\msl 2,3 \msr$}
			[.{$\msl 3,4 \msr$}
				[.{$\msl 4,5 \msr$}
					[.{$\msl 5,6 \msr$} ]
					[.{$\msl 5,6 \msr$} ]
				]
				[.{$\msl 4,5 \msr$}
					[.{$\msl 5,6 \msr$} ]
					[.{$\msl 5,6 \msr$} ]
				]
            ]
            [.{$\msl 3,4 \msr$}
				[.{$\msl 4,5 \msr$}
					[.{$\msl 5,6 \msr$} ]
					[.{$\msl 5,6 \msr$} ]
				]
				[.{$\msl 4,5 \msr$}
					[.{$\msl 5,6 \msr$} ]
					[.{$\msl 5,6 \msr$} ]
				]
			]
		]
		[.\textcolor{red}{V1}
			[.\textcolor{red}{V2}
				[.{$\msl 1,2 \msr$}
					[.{$\msl 2,3 \msr$} ]
					[.{$\msl 2,3 \msr$} ]
				]
				[.{$\msl 1,2 \msr$}
					[.{$\msl 2,3 \msr$} ]
					[.{$\msl 2,3 \msr$} ]
				]
            ]
            [.{$\msl 1,2 \msr$}
				[.{$\msl 2,3 \msr$}
					[.{$\msl 3,4 \msr$} ]
					[.{$\msl 3,4 \msr$} ]
				]
				[.{$\msl 2,3 \msr$}
					[.{$\msl 3,4 \msr$} ]
					[.{$\msl 3,4 \msr$} ]
				]
			]
		]
	 ]
\end{tikzpicture}
\caption{The multiset representations of the vertices with respect to the set of red-coloured vertices $V1$ and $V2$ are given as labels of each vertex. Notice that the class with multiset $\msl 1,2 \msr$ has cardinality $4$, while the remaining classes have more than $4$ vertices.}\label{fig:T_4}
\end{figure}

\subsection{The case of wheel graphs}

A \textit{wheel graph} $W_{1,n-1}$ is obtained from a cycle $C_{n-1}$ of order $n-1$ and an extra singleton vertex $v$, by adding an edge between $v$ and all the vertices of the cycle. The vertex $v$ is usually called the center of the wheel. We first observe that $W_{1,n-1}$ has the maximum degree $\Delta=n-1=|V(W_{1,n-1})|-1$. Thus, from Proposition \ref{prop:trivial-bounds}, it follows that $\kappa(W_{1,n-1})=n-1$. In this subsection, we consider the vertex set of the wheel graph $W_{1,n-1}$ to be $V(W_{1,n-1}) = \{0\} \cup [n-1] = \{0, 1, \ldots, n-1\}$, where the vertex $0$ has degree $n - 1$, and the vertices in $[n-1]$ induce the cycle $C_{n-1}$, with adjacency defined naturally.

\begin{lemma}\label{CentralBelongToAll}
Let $n\ge 7$. If $S$ is a $k$-MARS of the wheel graph $W_{1,n-1}$ for some $k\ge 4$, then $0\in S$.
\end{lemma}

\begin{proof}
Suppose, for the sake of contradiction, that $0 \notin S$. If $|S| = 1$, then there is an equivalence class with the multiset representation $\msl 1 \msr$ containing only $3$ vertices, which contradicts the fact that $k \geq 4$. If $|S| = 2$ and $S$ is formed by two vertices $i,j\in [n-1]$ of the cycle $C_{n-1}$ with $j=i+2$, then there is an equivalence class with the multiset representation $\msl 1,1 \msr$ containing only $2$ vertices ($0$ and the vertex $i+1$), which contradicts the fact that $k \geq 4$. Hence, assume that either $|S| \ge 3$, or $|S| = 2$ and $S$ is not composed of two vertices $i, j \in [n-1]$ such that $j = i + 2$. In this case, we have an equivalence class with multiset $\msl 1,\dots, 1\msr$ (with $|S|$ entries all equal to $1$) whose only vertex is $0$, which is a contradiction again. Therefore, $0$ belongs to any $k$-MARS of $W_{1,n-1}$.
\qed\end{proof}

The preceding lemma will assist in determining the $k$-multiset antidimension of wheel graphs for all suitable $k$. Observe first that $\adimms_{n-1}(W_{1,n-1}) = 1$, and the remaining cases are analyzed below.

\begin{theorem}
Let $W_{1,n-1}$ be a wheel graph of order $n\ge 7$. If $k\equiv 0\pmod 2$, then for any $4\le k< n-1$,
$$\adimms_k(W_{1,n-1})=\left\{\begin{array}{ccl}
   \dfrac{k+2}{2}; & \;\;\; &\text{if $n\ge\dfrac{5k+2}{2}$}, \\
    n-k; & \;\;\; &\text{if $\dfrac{3k+2}{2}\le n<\dfrac{5k+2}{2}$}, \\
    \infty; & \;\;\; &\text{otherwise}.
\end{array}\right.$$
If $k\equiv 1 \pmod 2$, then for any $5\le k< n-1$,
$$\adimms_k(W_{1,n-1})=\left\{\begin{array}{ccl}
   \min\left\{\left\lfloor\dfrac{n-1-k}{3}\right\rfloor+t+1,\dfrac{3k+3}{2}\right\}; & \;& \text{if $n\ge\dfrac{5k+5}{2}$}, \\
    n-k; & \;& \text{if $2k< n<\dfrac{5k+5}{2}$}, \\
    \infty; & \;& \text{otherwise},
\end{array}\right.$$
where $t$ is the remainder of dividing $(n-1-k)$ by $3$.
\end{theorem}

\begin{proof}
In connection with the notation already defined for the vertices of the wheel, when we mention a gap between two consecutive vertices of a $k$-MARS $S$, we are specifically referring to two consecutive (with respect to the natural labeling of the cycle $C_{n-1}$) vertices of $S\cap (V(W_{1,n-1})\setminus\{0\})$.

We first need to keep the following facts in mind. Based on Lemma~\ref{CentralBelongToAll} and the fact that $D(W_{1,n-1}) = 2$, any $k$-MARS of $W_{1,n-1}$ has at most three equivalence classes. For each $i \in \{1, 2, 3\}$, we define $V_i$ as the class of vertices whose multiset representations contain exactly $i$ entries equal to $1$, with all other entries equal to $2$. Moreover, since $k < n - 1$, at least one of the equivalence classes must contain exactly $k$ vertices, while the other two are either empty or contain at least $k$ vertices. Consequently, we have $\sum_{i=1}^3 |V_i| \ge k$. Observe also that in any $k$-MARS of $W_{1,n-1}$, a gap of at least two vertices between consecutive selected vertices necessarily contains exactly two vertices from $V_2$. Furthermore, if $V_1 \ne \emptyset$ for some $k$-MARS $S$, then there must exist a gap of at least three vertices between two consecutive vertices in $S$, which implies that $V_2 \ne \emptyset$.

First, we will assume that $k\equiv 0\pmod 2$. Let $B$ be a $k$-MARB of $W_{1,n-1}$. Observe that $B\ne \{0\}$ because $\{0\}$ is an $(n-1)$-MARB, and we have assumed $k<n-1$. Also, note that each vertex in $B\cap (V(W_{1,n-1})\setminus \{0\})$ has at most two neighbors belonging to the same equivalence class, $V_1$ or $V_2$. In addition, the fact that $1\leq|B\cap(V(W_{1,n-1})\setminus\{0\})|<n-1$ means that $V_1$ or $V_2$ are non-empty. Furthermore, considering that $|V_1|\geq k$ or $ |V_2|\geq k$, we conclude that $\adimms_k(W_{1,n-1})=|B|\ge\dfrac{k}{2}+1=\dfrac{k+2}{2}$. We analyze the following three cases.

\medskip
\noindent
\textbf{Case 1:} $n=\dfrac{3k+2}{2}$ or $n\ge\dfrac{5k+2}{2}$. Let $\displaystyle S=\{0\}\cup \bigcup_{i=1}^{\frac{k}{2}}(3i-2)$ be a subset of vertices of $W_{1,n-1}$. We claim that $S$ is a $k$-MARS of $W_{1,n-1}$. If $n=\dfrac{3k+2}{2}$, then we have $|V_2|=k$ and $V_1=V_3=\emptyset$, which leads to confirm that $S$ is a $k$-MARS of $W_{1,n-1}$. If $n\ge\dfrac{5k+2}{2}$, then we deduce $|V_2|=k$, $V_3=\emptyset$ and $|V_1|\ge k$ and consequently, $S$ is also a $k$-MARS of $W_{1,n-1}$. Therefore, $\adimms_k(W_{1,n-1})\le |S|=\dfrac{k+2}{2}$ from which it follows that $\adimms_k(W_{1,n-1})=\dfrac{k+2}{2}$.

\medskip
\noindent
\textbf{Case 2:} $\dfrac{3k+2}{2}\le n<\dfrac{5k+2}{2}$. Since $n=|B|+\sum_{i=1}^3|V_i|\ge \dfrac{k+2}{2}+\sum_{i=1}^3|V_i|$ and $n<\dfrac{5k+2}{2}$, it follows that $\sum_{i=1}^3|V_i|<2k$. Since at least one of these three equivalence classes must contain exactly $k$ vertices, the other two equivalence classes must necessarily be empty. Note that if $|V_1|=k$, then $V_2\ne\emptyset$, which leads to a contradiction. Therefore, we are left with two possibilities: either $|V_2|=k$ and $V_1=V_3=\emptyset$, or $|V_3|=k$ and $V_1=V_2=\emptyset$. In both cases it is straightforward to check that $\adimms_k(W_{1,n-1})=|B|=n-k$. In the first scenario, for each of the $k$ vertices in $V_2$, at least $\dfrac{k}{2}$ vertices from $B\cap(V(W_{1,n-1}) \setminus \{0\})$ are required. Hence, this scenario can only occur if $n\ge\dfrac{3k+2}{2}$. In the second scenario, for each of the $k$ vertices in $V_3$, at least $k$ vertices from $B\cap(V(W_{1,n-1}) \setminus \{0\})$ are needed. Therefore, this case can only happen if $n\ge\dfrac{4k+2}{2}$. Finally, we highlight the fact that if $n=\dfrac{3k+2}{2}$, then $\adimms_k(W_{1,n-1})=\dfrac{k+2}{2}=n-k$. Hence, we now consider the case where $n=\dfrac{3k+2}{2}$ in Case 2, rather than in Case 1.

\medskip
\noindent
\textbf{Case 3:} $k+1<n<\dfrac{3k+2}{2}$. Since $n=|B|+\sum_{i=1}^3|V_i|\ge\dfrac{k+2}{2}+\sum_{i=1}^3|V_i|$, we have $\sum_{i=1}^3|V_i|<k$, which leads to a contradiction. Therefore in this case there is no $k$-MARS of $W_{1,n-1}$.

\medskip
Taking into account the previous three cases and the fact that $n-k<\dfrac{k+2}{2}$ when $n<\dfrac{3k+2}{2}$, we have derived the result for $k\equiv 0\pmod 2$.

\bigskip
From now on assume that $k\equiv 1 \pmod 2$. As can be inferred from the aforementioned facts, $|V_2|\equiv 0\pmod 2$. Therefore, in this case, if $V_2\ne\emptyset $, then we have $|V_2|\ge k+1$. As a result, we have the following five cases for equivalence classes, each of which will yield the $k$-MARS with the minimum cardinality for that case.

\medskip
\noindent
\textbf{Case 4:} $|V_3|=k, |V_1|\ge k, |V_2|\ge k+1$. Let $S$ be a $k$-MARS of $W_{1,n-1}$ with minimum cardinality under these conditions. Since the $k$ vertices of $V_3$ have two neighbors in $S\cap(V(W_{1,n-1})\setminus \{0\})$ and $V_1\ne\emptyset$, there are at least $k+1$ neighboring vertices of $V_3$ in $S\cap(V(W_{1,n-1})\setminus \{0\})$. On the other hand, since there are exactly two vertices from $V_2$ in each gap of at least two vertices between two consecutive vertices in $S$, we can conclude that at least $\dfrac{k-1}{2}$ additional vertices must be part of $S\cap(V(W_{1,n-1})\setminus\{0\})$. Thus $|S|\ge (k+1)+\dfrac{k-1}{2}+1=\dfrac{3k+3}{2}$ and $\displaystyle n=|S|+\sum_{i=1}^3|V_i|\ge\dfrac{9k+5}{2}$. Let $\displaystyle S'=\{0\}\cup\bigcup_{i=1}^{k+1}(2i-1)\cup\bigcup_{i=1}^{\frac{k-1}{2}}(2k+3i+1)$ be a subset of vertices of $W_{1,n-1}$. We claim that $S'$ is a $k$-MARS of $W_{1,n-1}$. To this end, we can observe that $S'$ contains exactly $k$ gaps of one vertex and $\dfrac{k+1}{2}$ gaps of at least two vertices. Given that $n\ge\dfrac{9k+5}{2}$, one of these gaps with at least two vertices must contain at least $k+2$ vertices. Therefore, we have $|V_3| = k$, $|V_1| \ge k$, and $|V_2|=k+1$ with respect to $S'$. As a consequence, $S'$ is a $k$-MARS under these conditions $|S|\le |S'|=\dfrac{3k+3}{2}$, and thus, if $n\ge\dfrac{9k+5}{2}$ we conclude that $|S|=\dfrac{3k+3}{2}$.

\medskip
\noindent
\textbf{Case 5:} $|V_1|=k, |V_3|=0, |V_2|\ge k+1$. Let $S$ be a $k$-MARS of $W_{1,n-1}$ with minimum cardinality under these conditions. Since $|V_3| = 0$, each vertex $u \in V_2$ has exactly one consecutive vertex $v \in V_2 \setminus\{u\}$ such that at least one vertex of $S$ lies between $u$ and $v$ in $S\cap(V(W_{1,n-1})\setminus\{0\})$. Furthermore, given that $|V_1|=k$, we deduce that $|S\cap(V(W_{1,n-1})\setminus\{0\})|\ge\left\lfloor\dfrac{n-1-k}{3}\right\rfloor+t$, where $t$ is the remainder when $(n-1-k)$ is divided by $3$. Thus, $|S|\ge\left\lfloor\dfrac{n-1-k}{3}\right\rfloor+t+1$ and $\displaystyle n=|S|+|V_1|+|V_2|\ge\left(\dfrac{|V_2|}{2}+1\right)+|V_1|+|V_2|\ge\dfrac{5k+5}{2}$. Let $\displaystyle S'=\{0\}\cup\bigcup_{i=1}^{\lfloor\frac{n-1-k}{3}\rfloor}(3i-2)\cup\bigcup_{i=1}^{t}\left(3\left\lfloor\frac{n-1-k}{3}\right\rfloor-2+i\right)$ be a subset of vertices of $W_{1,n-1}$. We claim that $S'$ is a $k$-MARS of $W_{1,n-1}$. With this goal, we can realize that $|V_1|=k$, $|V_2|=2\left\lfloor\dfrac{n-1-k}{3}\right\rfloor$ and $|V_3|=\emptyset$. Since $n\ge\dfrac{5k+5}{2}$, we have $|V_2|\ge k+1$, and as a consequence, $S'$ is a $k$-MARS under these conditions. Thus, $|S|\le |S'|=\left\lfloor\dfrac{n-1-k}{3}\right\rfloor+t+1$ and we conclude that if $n\ge\dfrac{5k+5}{2}$, then $|S|=\left\lfloor\dfrac{n-1-k}{3}\right\rfloor+t+1$.

\medskip
\noindent
\textbf{Case 6:} $|V_3|=k, |V_1|=|V_2|=0$. In this case, it is straightforward to see $|S|=n-k$, where $S$ is a $k$-MARS of $W_{1,n-1}$ with minimum cardinality under these conditions. Furthermore, since each vertex in $V_3$ requires at least one vertex in $S \cap (V(W_{1,n-1})\setminus\{0\})$, we have $n = |S| + |V_3| \ge (k+1) + k = 2k + 1$.

\medskip
\noindent
\textbf{Case 7:} $|V_3|=k, |V_1|=0, |V_2|\ge k+1$. Let $S$ be a $k$-MARS of $W_{1,n-1}$ with minimum cardinality under these conditions. Note that each vertex in $V_3$ needs at least one vertex in $S \cap (V(W_{1,n-1})\setminus\{0\})$. Furthermore, for the remaining $n-1-2k$ vertices in $V(W_{1,n-1}) \setminus \{0\}$, we require that for every 2 vertices in $V_2$, there is at least one vertex in $S\cap (V(W_{1,n-1}) \setminus \{0\})$. Therefore, $|S|\ge k+\left\lfloor\dfrac{n-1-2k}{3}\right\rfloor+t+1=\left\lfloor\dfrac{n-1+k}{3}\right\rfloor+t+1$, where $t$ is the remainder when $(n-1-k)$ is divided by $3$. Since $n=|S|+|V_2|+|V_3|\ge \left(k+\dfrac{|V_2|}{2}+1\right)+|V_2|+|V_3|\ge \dfrac{7k+3}{2}$, $S$ cannot serve as a $k$-MARB of $W_{1,n-1}$, as its cardinality will always exceed that of Case 2.

\medskip
\noindent
\textbf{Case 8:} $|V_1|=k, |V_3|\ge k, |V_2|\ge k+1$. In this case, following a procedure similar to that in Case 4, we conclude that for any $k$-MARS $S$ with minimum cardinality under these conditions, the inequality $|S| \ge \dfrac{3k+3}{2}$ holds for $n \ge \dfrac{9k+5}{2}$. Thus, in this case, we can always select a minimum cardinality $k$-MARS $S'$ that satisfies the conditions of Case 4.

\medskip
Considering the previous 5 cases and the fact that $\dfrac{5k+5}{2} < \dfrac{9k+5}{2}$, as well as $\left\lfloor\dfrac{n-1-k}{3}\right\rfloor +t+1\le\dfrac{3k+3}{2}$ when $n\le\dfrac{11k+11}{2}$, we deduce the desired result for $k \equiv 1 \pmod{2}$.
\qed\end{proof}

\section{Mathematical optimization formulations}

A discrete optimization model for the $k$-metric antidimension of graphs was previously presented in \cite{Fernandez2023}, which was used there to make several implementations on some classical families of graphs, as well as, on some random graphs. In this section, related formulations for the $k$-MAD are developed.

\subsection{Formulation with general integer variables}

The formulation is based on the definition of a $k$-MARS $S$, through its equivalence classes. It is built over several sets of binary decision variables and one set of general integer variables.
The set of binary variables $\mathbf s$ determines the elements of the set $S$. The set of binary variables $\mathbf q$ determines the elements of the different classes.
The vertex classes are determined by subsets of vertices that jointly satisfy some compatibility conditions, and will be referred to as $Q$-\emph{subsets}. To avoid multiple representations of the same solution, each $Q$-subset has a unique representative, which is its lowest index vertex.\\
Moreover, we use a set of general integer variables, $\mathbf t$ that ``count'' the number of vertices in the set $S$ at a fixed distance from a given vertex.  We use the notation $R=\{1, \dots, \bar d_{max}\}$ for the values of the possible pairwise distances;  indeed, the maximum potential distance, $\bar d_{max}$, is the diameter of the graph.

Hence, initially, we define the following sets of decision variables, although further sets of binary decision variables will be introduced later:
\begin{align}
s_u & = 1 \iff u \in S \qquad && \forall u\in V \nonumber \\
q_{uv} & = 1 \iff v\text{ is in the $Q$-subset with representative $u$} \qquad && \forall u,v\in V, v\geq u,\nonumber\\
t_{ur}: & \qquad\text{ number of vertices of } S\text{ at distance $r\in R$ from $u$} \qquad && \forall u\in V.\nonumber
\end{align}

With the above decision variables we can determine vertex sets $S=\{u \in V: s_u=1\}$ of cardinality $|S|=\sum_{u \in V} s_u$, and $Q$-subsets $Q^u=\{v\in V: q_{uv}=1\}$. Since $q_{uu}=1$ indicates that $u$ is the lowest index vertex of a $Q$-subset, the actual number of classes determined by the solution is $\overline r=\sum_{u\in V}q_{uu}$. The number of vertices of $S$ at distance $r$ from a given vertex $u$ can be expressed as $t_{ur}=\sum\limits_{\substack{v\in V: d_{uv}=r}} s_v$.

The (preliminary) formulation is as follows.

\begin{align}
F\qquad\hspace{-0.3cm}\min \quad &  \sum_{u \in V} s_u  && \\
& \sum_{u\in V} s_{u}  \ge 1 && \label{2ineq:S-non-empty}\\
& s_u+\sum_{v\in V: v\leq u} q_{vu}  = 1 && u\in V\label{2ineq:partition}\\ 
& q_{uv}\leq q_{uu} && u,v\in V,\, u\geq v\label{2ineq:logic1}\\
& s_u+\,q_{uv}  \leq 1  &&  u,v\in V,\, v\geq u \label{2ineq:incompatibility}\\
& \sum_{v\in V: v> u} q_{uv}  \ge (k-1) q_{uu} && u \in V \label{2ineq:cardinality}\\
& t_{ur}=\sum_{\substack{v\in V\\ d_{uv}=r}}s_v  && u \in V, r\in R \label{2ineq:count}\\
& t_{ur}\geq t_{vr}- M(1-q_{uv}) \qquad && u,v \in V,\, u\leq v \label{2ineq:def_class1}\\
& t_{vr}\geq t_{ur}- M(1-q_{uv}) \qquad && u,v \in V,\, u\leq v \label{2ineq:def_class2}\\
& s_u\in\{0, 1\}, \, &&  u\in V;\\
& q_{uv}\in\{0, 1\},\, &&  u, v \in V, v\geq u\\
& t_{ur}\in \mathbf{Z}^+,\, && u \in V, r\in R.
\end{align}


Constraints \eqref{2ineq:S-non-empty} impose that the set $S$ is not empty and Constraints \eqref{2ineq:partition}-\eqref{2ineq:incompatibility}  guarantee that we obtain a partition of the vertex set: every vertex either is in $S$ or in a single class, whose representative is not an element of $S$.
By Constraints \eqref{2ineq:cardinality}, the cardinality of each class is at least $k$. While the equalities \eqref{2ineq:count} establish the number of elements of $S$ at every distance of a given vertex $u$, the inequalities \eqref{2ineq:def_class1}-\eqref{2ineq:def_class2} impose that the number of vertices at distance $r$ from $S$ is the same as that of $u$, for all the vertices in the class with representative $u$. The big-$M$ constant $M$, must be sufficiently big to deactivate the constraint when $q_{uv}=0$; we use the value $n-k$.\\

The above formulation has $\mathcal O(|V|)$ binary variables $\mathbf s$ and $\mathcal O(|V|^2)$ binary variables $\mathbf q$. The number of general integer decision variables $\mathbf t$ is $\mathcal O(|V||R|)$. The number of constraints is dominated by $\mathcal O(|V|^2)+\mathcal O(|V||R|)$.

Unfortunately, the above formulation does not guarantee that classes are \emph{maximal}. That is, it may produce two classes with exactly the same number of vertices at distance $r$ from $S$, for all values of $r\in R$. In the previous work \cite{Fernandez2023} this was avoided by imposing that if two vertices $u, v$ are not in the same class (and they were not in $S$), there is at least one vertex in $w\in S$ such that $d_{uw}\ne d_{vw}$. Of course, this condition is no longer valid. Instead, what we want to model is that if $u, v$ are not in the same class (and they are not in $S$), then there must be some $r\in R$ such that
$t_{ur}\ne t_{vr}$. This condition is difficult to model because it may happen that $t_{ur}>t_{vr}$, or $t_{ur}<t_{vr}$. 
Still, we observe that for all $u\in V$, $\sum_{r\in R}t_{ur}=|S|$ is constant. Hence, when for a given pair, $u,v\in V$, if $t_{ur}< t_{vr}$ for some $r\in R$, then there must exist a different $r'\in R\setminus\{r\}$ such that $t_{ur'}> t_{vr'}$.

Hence, for any two classes, with representatives $u$ and $v$, respectively ($u\ne v$),  we will impose that that $t_{ur}> t_{vr}$ for some $r\in R$. For this, we introduce another set of binary decision variables, that we denote by $\delta_{uvr}$, $u,v\in V, r\in R$
such that
\begin{align}
\delta_{uvr}=1 & \iff t_{ur}>t_{vr}.  \qquad && \nonumber
\end{align}

The following set of constraints are needed to ensure that the new variables model the explained logic:

\begin{align}
& \delta_{uvr}+q_{uv}\leq 1 && u,v\in V, u<v,\, r\in R\label{logigc1}\\
& q_{uu}+q_{vv}\leq 1 + \sum_{r\in R}\delta_{uvr}\qquad && u,v\in V,\, u<v\,  \label{ineq:q-delta}\\
& 	\delta_{uvr}\leq t_{ur}-t_{vr}+(n-k)[1-\delta_{uvr}] \qquad && u,v\in V,\, u<v,\, r\in R\label{2ineq:last0}\\
& 	t_{ur}-t_{vr}\leq (n-k)[1+s_u-q_{uv}]\qquad && u,v\in V,\, u<v,\, r\in R\label{2ineq:last1}\\
& 	t_{vr}-t_{ur}\leq (n-k)[1+s_u-q_{uv}]\qquad && u,v\in V,\, u<v,\, r\in R\label{2ineq:last2}
\end{align}

By \eqref{logigc1}, the differences $\delta_{uvr}$ are only computed for vertices in different classes.
Constraints \eqref{ineq:q-delta} impose that, when $u$ and $v$ are representatives of different classes, their respective classes must have a different number of vertices at distance $r$ from $S$ for at least one value of $r\in R$.
The right hand side of \eqref{2ineq:last0} guarantees that when $\delta_{uvr}=1$, then $1\leq t_{ur}-t_{vr}$, i.e.  $t_{vr}<t_{ur}$. Note that when $\delta_{uvr}=0$ the constraint is trivially satisfied as the difference $t_{vr}-t_{ur}$ is never greater than $n-k$, which is an upper bound on the number of vertices in any $Q$-class.
The reverse condition is imposed in \eqref{2ineq:last1}-\eqref{2ineq:last2}. When  $t_{ur}>t_{vr}$, then \eqref{2ineq:last1} imposes that $1+s_u-q_{uv}\ne 0$, i.e. $v$ cannot belong to the class represented by $u$. A similar logic is derived from \eqref{2ineq:last2} when  $t_{vr}>t_{ur}$.\\

The size of the formulation has increased in $\mathcal O(|V|^2|R|)$ binary decision variables and a number of constraints, dominated by $\mathcal O(|V|^2|R|)$.\\

It is indeed possible to develop an alternative formulation using binary variables only. In particular,
for all $u\in V$, $r\in R$, $h\in H=\{0, \dots, n-k\}$ we may introduce a binary decision variable:
$$z_{urh}=1\iff\text{there are $h$ vertices of $S$ at distance $r$ from $u$}.$$
Then, we have $$t_{ur}=\sum_{h\in H}h\, z_{urh},$$
provided that $s_u+\sum_{h\in H}z_{urh}=1$.
Hence, we can eliminate the general integer variables $t_{ur}$ and obtain a formulation with binary variables only.

While the advantage of such a model is that it produces tight linear programming bounds, it increases the number of decision variables in $\mathcal{O}({|V||R||H|})$. Preliminary testing indicates that the memory requirements of such a formulation only allow to solve small-size instances. Hence, this alternative is not further explored.

\section{Implementations}
In this section we present the numerical results of the computational experiments we have carried out with the formulation presented in the previous section. 


All the computational tests have been carried out in an AMD Ryzen 7 PRO 2700U 2.20 GHz with 8 GB RAM, under Windows 10 Pro as operating system. Formulation $F$  has  been coded in Mosel 5.6.0 using as solver Xpress Optimizer Version 38.01.01  \cite{Xpress}.

For the experiments we have considered the following sets of benchmark instances, which were already used in \cite{Fernandez2023}:
\begin{itemize}
\item $\mathcal{S}$ instances: General sparse graphs with a number of vertices $n\in\{50, 100\}$ and maximum vertex degree $\delta\in\{5, 10\}$ for $n=50$, and $\delta\in \{5, 10, 15, 20\}$, for $n=100$.
Originally, directed graphs are generated and then all directions removed.
The arcs of the graph are generated by iteratively \emph{exploring} its vertices and randomly generating up to $\delta$ end-nodes (from the original vertex set) for the arcs with origin at the current vertex. For each combination of $n$ and $\delta$ two instances have been generated.
\item $\mathcal{D}$ instances: General dense graphs with a number of vertices $n\in\{50, 100\}$, and vertex degree $\delta\in\{40, 45\}$ for $n=50$, and $\delta\in\{90, 95\}$ for $n=100$. 
For each combination of $n$ and $\delta$ two instances have been generated by removing $\delta$ randomly generated edges from the complete graph $K_n$.
\end{itemize}

Each instance of the classes $\mathcal{S}$ and $\mathcal{D}$ has been solved for all values of $k\in\{1, 2, 3, 4,  5, 6\}$. For these instances, a time limit of 7,200 seconds was set for each run.
The results are summarized in Tables \ref{Tab:S} and \ref{Tab:D} for sparse ($\mathcal S$) and dense ($\mathcal D$) graphs, respectively.
The tables show the values of the instance parameters in the first two columns, followed by $k$ blocks, with two columns each, corresponding to the considered values of $k$. The first column in each block, labeled with $|S|$ gives the values of the $k$-multiset antidimension of the considered graph ($\adimms_k(G)$). 
The second column in each block, labeled with $\emph{CPU}$ gives the computing time required by the solver to obtain a provable optimal solution, or \emph{``TL''} when the time limit was reached. When the time limit was reached, the entry under column $|S|$  gives the value of the best solution found (although the optimality of such solution was not proven), or \emph{``-''} when no feasible solution was found.
The tables have one row for each of the two instances with the same characteristics.
%

\begin{table}[ht]
\centering
\scriptsize
\begin{center}
\begin{tabular}{|cc|cr|cr|cr|cr|cr|cr|}												
\hline
\multirow{2}{*}{$n$} & \multirow{2}{*}{$\delta$} & \multicolumn{2}{c|}{$k=1$}  & \multicolumn{2}{c|}{$k=2$}  & \multicolumn{2}{c|}{$k=3$}  & \multicolumn{2}{c|}{$k=4$} & \multicolumn{2}{c|}{$k=5$} & \multicolumn{2}{c|}{$k=6$}\\
\cline{3-14}												
 &&	\multicolumn{1}{c}{$|S|$} & \multicolumn{1}{r|}{CPU} &	\multicolumn{1}{c}{$|S|$} & \multicolumn{1}{r|}{CPU}&	 \multicolumn{1}{c}{$|S|$} &	 \multicolumn{1}{r|}{CPU}&	 \multicolumn{1}{c}{$|S|$} &	 \multicolumn{1}{r|}{CPU} &	 \multicolumn{1}{c}{$|S|$} &	 \multicolumn{1}{r|}{CPU} &	 \multicolumn{1}{c}{$|S|$} &	 \multicolumn{1}{r|}{CPU}\\	
 \hline
\multirow{4}{*}{50} & \multirow{2}{*}{5}	    &   2 &   1.2$\,$	&  2  &    2.8$\,$	& -	&    TL$\,$   &	1 &   0.9$\,$  & - &   TL$\,$  & - & TL$\,$\\
                    & 	                        &   2 & 425.7$\,$	&  1  &    1.7$\,$	& -	&    TL$\,$   &	- &    TL$\,$  & - &   TL$\,$  & - & TL$\,$\\
                    & \multirow{2}{*}{$10\,$}	& 	2 &   2.2$\,$   &  2  & 1478.9$\,$  &  2 & 2569.2$\,$ & 2 &	4747.8$\,$ &  1 & 1.05$\,$ & - & TL$\,$\\
                    &                    	    & 	2 &   1,6$\,$   &  2  &    1.8$\,$  &  2 & 1523.5$\,$ & 3 &	1275.9$\,$ &  1	& 1.06$\,$ & - & TL$\,$\\	
\hline	
\multirow{8}{*}{100} & \multirow{2}{*}{5}	    &   2 &  11.5$\,$   &  -  &   TL$\,$    & -  &    TL$\,$  &  - &  TL$\,$   &  2 &  5.9$\,$ & - & TL$\,$\\
         	     &                         &   2 &  12,7$\,$   & -	 &   TL$\,$	   & 1	&   7.8$\,$  &  - &  TL$\,$	  &  - &   TL$\,$ & - & TL$\,$\\
                     &  \multirow{2}{*}{$10\,$} &   2 &  10.1$\,$   &  -	 &   TL$\,$    & -	&    TL$\,$	 &  - &  TL$\,$   &  - &   TL$\,$ & - & TL$\,$\\
                     &                          &   2 &  10.3$\,$   &  98 &   TL$\,$	   &  -	&    TL$\,$	 &  - &  TL$\,$	  &  - &   TL$\,$ & - & TL$\,$\\
                     & \multirow{2}{*}{$15\,$}  &   2 &  10.2$\,$   &  2	 &  8.7$\,$	   & 2	& 5154.9$\,$ &  - &  TL$\,$	  &  - &   TL$\,$ & 1 & 6.9$\,$\\
                      &                         &	  2	&  10.2$\,$	  &  2	&  9.2$\,$	  & -  &     TL$\,$ &  - &  TL$\,$	 &  1 &  8.1$\,$ & - & TL$\,$\\
                    &  \multirow{2}{*}{$20\,$}  &  87 &    TL$\,$	&  11 &    TL$\,$   & 2	 &    7.8$\,$  &  2 &  6.7$\,$   & 89	&  	TL$\,$ & - & TL$\,$\\
                    &                           &  2  & 199.0$\,$	&  2  & 640.0$\,$   & -	 &    TL$\,$  &  - &  TL$\,$   &  -	&  	TL$\,$ & - & TL$\,$\\
\hline
\end{tabular}
\caption{Summary of results for $\mathcal{S}$ instances: sparse graphs with $n$ vertices and vertex degree $\delta$.}\label{Tab:S}
\end{center}
\end{table}

\begin{table}[ht]
\centering
\scriptsize
\begin{center}
\begin{tabular}{|cc|cr|cr|cr|cr|cr|cr|}												
\hline
\multirow{2}{*}{$n$} & \multirow{2}{*}{$\delta$} & \multicolumn{2}{c|}{$k=1$}  & \multicolumn{2}{c|}{$k=2$}  & \multicolumn{2}{c|}{$k=3$}  & \multicolumn{2}{c|}{$k=4$} & \multicolumn{2}{c|}{$k=5$} & \multicolumn{2}{c|}{$k=6$}\\
\cline{3-14}												
 &&	\multicolumn{1}{c}{$|S|$} & \multicolumn{1}{r|}{CPU} &	\multicolumn{1}{c}{$|S|$} & \multicolumn{1}{r|}{CPU}&	 \multicolumn{1}{c}{$|S|$} &	 \multicolumn{1}{r|}{CPU}&	 \multicolumn{1}{c}{$|S|$} &	 \multicolumn{1}{r|}{CPU} &	 \multicolumn{1}{c}{$|S|$} &	 \multicolumn{1}{r|}{CPU} &	 \multicolumn{1}{c}{$|S|$} &	 \multicolumn{1}{r|}{CPU}\\	
 \hline
\multirow{4}{*}{50} & \multirow{2}{*}{40}	    &   2 &   58.6$\,$	&  2  &    372.3$\,$	& 2	&   391.8$\,$   & 1  &    0.9$\,$  & 1 &   0.6$\,$  & 1 & 0.5$\,$\\
                    & 	                        &   2 &  139.9$\,$	&  2  &    225.8$\,$	& 2	&    34.4$\,$   & 1  &    1.9$\,$  & 1 &   6.1$\,$  & 1 & 2.2$\,$\\
                    & \multirow{2}{*}{$45\,$}	& 	2 &  470.7$\,$  &  2  &    370.6$\,$    &  2 & 1018.2$\,$   & 2  &	470.8$\,$  & 2 & 162.9$\,$  & 2 & 283.8$\,$\\
                    &                    	    & 	2 &  609.0$\,$  &  2  &    899.1$\,$    &  2 &  186.2$\,$   & 2  &	579.1$\,$  & 2 & 261.4$\,$  & 2 & 167.5$\,$\\	
\hline	
\multirow{4}{*}{100} & \multirow{2}{*}{90}	    &   2 &  2096.5$\,$   &  2  &   TL$\,$    & -  &    TL$\,$  &  1 &  5.3$\,$   &  - &   TL$\,$ & - & TL$\,$\\
         	     &                         &   2 &  1680.4$\,$   & 73  &   TL$\,$	 & -  &    TL$\,$  &  1 & 17.6$\,$	 &  1 &  4.3$\,$ & 7 & TL$\,$\\
                     &  \multirow{2}{*}{$95\,$} &  87 &      TL$\,$   &  -  &   TL$\,$    & 6  &    TL$\,$  &  - &   TL$\,$   &  - &   TL$\,$ & - & TL$\,$\\
                     &                          &  92 &      TL$\,$   &  -  &   TL$\,$	 & 2  &    TL$\,$  &  - &   TL$\,$	 &  - &   TL$\,$ & - & TL$\,$\\
\hline
\end{tabular}
\caption{Summary of results for $\mathcal{D}$ instances: dense graphs with $n$ vertices and degree $\delta$.}\label{Tab:D}
\end{center}
\end{table}

From the above tables, it seems clear that, independently of whether the considered graph is sparse or dense, and independently of the value of the parameter $k$, instances with an optimal value $|S|=1$, are \textit{easy} to solve, as they can be solved within a few seconds. 
This behavior changes remarkably for instances with $|S|>1$, which are notably more challenging to solve.
As can be seen, with very few exceptions, both for sparse and dense graphs, instances with $n=100$ could not be solved to proven optimality, nor their infeasibility proven, within the maximum computing time.
For most of such instances, no feasible solution was found within the time limit.  Moreover, in the (few) cases in which a feasible solution was found, $|S|$ is usually very high, suggesting that such solutions are far from optimal.
If we restrict our analysis to instances with $n=50$, which could be optimally solved in most cases, the results in Table \ref{Tab:S}, indicate that for sparse graphs, the difficulty in finding optimal solutions increases with the value of the parameter $k$. In particular, for sparse graphs with $k=6$ no feasible solution was found for any of the tested instances. It is possible that no feasible solution exists for these instances, although their infeasibility could neither be proven or disproven within the time limit.
The behavior with sparse instances is opposite to that of dense graphs, where all instances with $n=50$ could be solved to proven optimality, and the difficulty for solving the instances seems to decrease as $k$ increases.

We also observe that with one single exception, all instances that could be solved to proven optimality within the time limit, have an optimal value $|S|\in\{1, 2\}$. The exception corresponds to an sparse graph instance with $n=50$ and $k=4$, whose optimal value is $|S|=3$. In addition, we also observe that no instance with $k=2$ proved to be certainly unfeasible. This contributes to our idea concerning the smallest possible value for $\kappa(G)$. That is, we suspect that in general any graph $G$ satisfies that $\kappa(G)\ge 2$.

We close this experiment analysis with the results obtained with two cycle instances, one with $37$ vertices, denoted as $\mathcal C_{37}$, and the other one with $40$ vertices, denoted as $\mathcal C_{40}$. These two instances have been solved for all the values of the parameter $k$ ranging in $\{1, 2, \dots, 37\}$. A time limit of $3,600$ seconds was set for each run. As expected, the infeasibility of most instances was detected by the solver within the time limit, and only in a few cases the solver found a provable optimal solution or was not able to conclude that the instance was infeasible.  The results are summarized in Table \ref{Tab:C} where, we only show the values of the parameter $k$ for which the infeasibility was not established for at least one of the two instances.
For the instances for which an optimal solution was found, the column under $|S|$ gives the optimal value, and the column under CPU the computing time needed by the solver to prove optimality. Otherwise, the time limit was reached without finding any feasible solution. This is indicated with an entry  with an entry $``-''$ in column $|S|$ and an entry ``TL'' in column CPU.
Empty entries correspond to infeasible instances for the corresponding value of $k$.
\addtolength{\tabcolsep}{5pt}
\begin{table}[ht]
\centering
\scriptsize
\begin{center}
\begin{tabular}{|c|cr|cr|}												
\hline
\multirow{2}{*}{$k$} & \multicolumn{2}{c|}{$\mathcal C_{37}$} & \multicolumn{2}{c|}{$\mathcal C_{40}$}\\
\cline{2-5}
& \multicolumn{1}{c}{$|S|$}& \multicolumn{1}{c|}{CPU}& \multicolumn{1}{c}{$|S|$}& \multicolumn{1}{c|}{CPU}\\
\hline
1&  2 & 2.6 & 1 & 2.9\\
2&  1 & 2.2 & 3 & 3.2   \\
3&  - & TL&  - & TL              \\
4&    - & TL & 4 & 132.0         \\
5&  - & TL & 5 & 995.9       \\
6&    & & - & TL            \\
7&    & & - & TL  \\
8&     & & 8 & 394.5      \\
10&    & & 10 & 252.2    \\
16&   & & 8 & 60.7 \\
20&  & & 20 & 102.2\\
\hline
\end{tabular}
\caption{Summary of results for $\mathcal C_{37}$ and $\mathcal C_{40}$. Only values of $k$ for which at least one of the two instances was not proven infeasible are shown. }\label{Tab:C}
\end{center}
\end{table}

\subsection{The $(k, \ell)$-multiset anonymity met by the social graphs of the experiments}

The computational results described above show that, in general, graphs randomly generated are usually satisfying a very small privacy features with respect to active attacks to its privacy, under the assumption of the existence of one or two attacker vertices ($\ell = 1$ or $\ell = 2$). In addition, the low efficiency of the computations are weakly contributing to get some strong conclusions on this regard. In particular, we conclude that the main part of the graphs used are satisfying $(2, 2)$-anonymity (in the presence of $2$ attacker nodes).

The results obtained show a slightly profit of the $(k, \ell)$-multiset anonymity with respect to its predecessor the $(k, \ell)$-anonymity, which is somehow not surprising since the $(k, \ell)$-multiset anonymity better catches the real behavior of privacy properties of social graphs, with respect to our settings on active attacks.

\section{Concluding remarks}

According to the presented results, we next describe some possible open question that might be of interest for future researches in this direction.

\begin{itemize}
    \item The case of cycles $C_n$ has been surprisingly challenging in our investigation. In this sense, the following question might be of interest: Is it true that for an integer $n\ge 4$ the cycle $C_n$ has a $k$-MARS if and only if $k$ divides $n$, or $k$ divides $2n$? Notice that the results from the implementations on cycles (see Table \ref{Tab:C}) support a positive answer for this.
    \item In contrast with the classical $k$-antiresolving set (using vectors of distances), it seems that every graph has $k$-MARS for every $k\in \{1,2\}$. Notice that there are several graphs $G$ which are $1$-metric antidimensional, which are those graphs that only have $1$-antiresolving sets. Such graphs represent the worst case scenario for privacy properties with respect to the $(k,\ell)$-anonymity measure (see \cite{cicerone2025,Trujillo2016} for results in this direction). We have proved for example that $\kappa(T)\ge 2$ for any tree $T$, and moreover, the results of the implementations show no example of a graph $G$ with $\kappa(G)=1$. Hence: Is it true that $\kappa(G)\ge 2$ for any connected graph $G$?
    \item Theorem \ref{th:trees-kappa} and the comments before it suggest considering the characterization of the class of graphs $G$ such that $\kappa(G)=2$, or at least the class of trees satisfying such property.
    \item A few partial results on the $k$-multiset antidimension of binary trees have been presented in Subsection \ref{subsec:computing}. In this sense, it would be desirable to complete such study, and indeed to generalize it to at least all $n$-ary trees.
    \item Although it might be natural to think that finding the $k$-multiset antidimension of graphs is an NP-hard problem, making the formal study of this  complexity problem seems to be worthwhile.
\end{itemize}

\section*{Acknowledgments}

This research has been partially supported by the ``HERMES'' project and the INCIBE-URV Cybersecurity Chair funded by the European Union NextGenerationEU/PRTR via INCIBE. E. Fern\'andez, D. Kuziak, M. Mu\~noz-M\'arquez and I.\ G.\ Yero have been partially supported by ``Ministerio de Ciencia, Innovaci\'on y Universidades'' through the grants PID2019-105824GB-I00 and PID2023-146643NB-I00, and Cadiz University Research Program. D. Kuziak and I. G. Yero have been also partially supported by ``Plan Propio de Apoyo y Est\'imulo a la Investigaci\'on y la Transferencia, Programa Operativo FEDER Andaluc\'ia 2021--2027'', reference number FEDER-UCA-2024-A2-16. R. Trujillo-Rasua was supported by a Ramón y Cajal grant (RYC2020-028954-I) from the Spanish Ministry of Science and Innovation and the EU, as well as, by the project PROVTOPIA (PID2023-150098OB-I00), funded by MICIU/AEI/10.13039/501100011033 and FEDER (EU). Alejandro Estrada-Moreno was also supported by PROVTOPIA.

\medskip

\section*{Conflict of interest}
The authors do not have any financial or non financial interests that are directly or indirectly related to the work submitted for publication.

\section*{Data availability} 
No data was used for the research described in this paper.





\begin{thebibliography}{10}

\bibitem{Backstrom2007}
L. Backstrom, C. Dwork, and J. Kleinberg.
\newblock Wherefore art thou r3579x? anonymized social networks, hidden
  patterns, and structural steganography.
\newblock In {\em Proceedings of the 16th International Conference on World
  Wide Web}, WWW '07, page 181–190, New York, NY, USA, 2007. Association for
  Computing Machinery.

\bibitem{CHATTERJEE201953}
T. Chatterjee, B. DasGupta, N. Mobasheri, V. Srinivasan,
  and I.~G. Yero.
\newblock On the computational complexities of three problems related to a
  privacy measure for large networks under active attack.
\newblock {\em Theoretical Computer Science}, 775:53--67, 2019.

\bibitem{cicerone2025}
S. Cicerone, G. Di~Stefano, S. Klav{\v{z}}ar, and I.~G.
  Yero.
\newblock Burning some myths on privacy properties of social networks against
  active attacks, arXiv:2504.16944 [cs.SI], 2025.

\bibitem{DASGUPTA201987}
B. DasGupta, N. Mobasheri, and I.~G. Yero.
\newblock On analyzing and evaluating privacy measures for social networks
  under active attack.
\newblock {\em Information Sciences}, 473:87--100, 2019.

\bibitem{diaz2017}
J. D{\'\i}az, O. Pottonen, M. Serna, and E.~J. van Leeuwen.
\newblock Complexity of metric dimension on planar graphs.
\newblock {\em Journal of Computer and System Sciences}, 83(1):132--158, 2017.

\bibitem{Erfani2019}
S.~H. Erfani and R. Mortazavi.
\newblock A novel graph-modiﬁcation technique for user privacy-preserving on
  social networks.
\newblock {\em Journal of Telecommunications and Information Technology},
  3:27--38, 10 2019.

\bibitem{Fernandez2023}
E. Fern\'andez, D. Kuziak, M. Mu\~noz M\'arquez, and I.~G. Yero.
\newblock On the ($k$, $\ell$)-anonymity of networks via their $k$-metric
  antidimension.
\newblock {\em Scientific Reports}, 13(1), 2023.

\bibitem{GILPONS2019124612}
R. Gil-Pons, Y. Ramírez-Cruz, R. Trujillo-Rasua, and I.~G.
  Yero.
\newblock Distance-based vertex identification in graphs: The outer multiset
  dimension.
\newblock {\em Applied Mathematics and Computation}, 363:124612, 2019.

\bibitem{hakanen2024}
A. Hakanen and I.~G. Yero.
\newblock Complexity and equivalency of multiset dimension and id-colorings.
\newblock {\em Fundamenta Informaticae}, 191(3-4):315--330, 2024.

\bibitem{Harary1976}
F. Harary and R.~A. Melter.
\newblock On the metric dimension of a graph.
\newblock {\em Ars Combinatoria}, 2:191--195, 1976.

\bibitem{klavzar-outermult}
S. Klav{\v{z}}ar, D. Kuziak, and I.~G. Yero.
\newblock Further contributions on the outer multiset dimension of graphs.
\newblock {\em Results in Mathematics}, 78(2):50, 2023.

\bibitem{Kratica2019KmetricAO}
J. Kratica, V. Kova\v{c}evi\'c-Vuj\v{c}i\'c, and M.
  \v{C}angalovi\'c.
\newblock $k$-metric antidimension of some generalized petersen graphs.
\newblock {\em Filomat}, 2019.

\bibitem{kuziak2021}
D. Kuziak and I.~G. Yero.
\newblock Metric dimension related parameters in graphs: A survey on
  combinatorial, computational and applied results, arXiv:2107.04877 [math.CO], 2021.

\bibitem{MauwRT18}
S. Mauw, Y. Ram{\'{\i}}rez{-}Cruz, and R. Trujillo{-}Rasua.
\newblock Anonymising social graphs in the presence of active attackers.
\newblock {\em Trans. Data Priv.}, 11(2):169--198, 2018.

\bibitem{conditional}
S. Mauw, Y. Ram{\'{\i}}rez{-}Cruz, and R. Trujillo{-}Rasua.
\newblock Conditional adjacency anonymity in social graphs under active
  attacks.
\newblock {\em Knowl. Inf. Syst.}, 61(1):485--511, 2019.

\bibitem{MauwRT19}
S. Mauw, Y. Ram{\'{\i}}rez{-}Cruz, and R. Trujillo{-}Rasua.
\newblock Robust active attacks on social graphs.
\newblock {\em Data Min. Knowl. Discov.}, 33(5):1357--1392, 2019.

\bibitem{MauwRT22}
S. Mauw, Y. Ram{\'{\i}}rez{-}Cruz, and R. Trujillo{-}Rasua.
\newblock Preventing active re-identification attacks on social graphs via
  sybil subgraph obfuscation.
\newblock {\em Knowl. Inf. Syst.}, 64(4):1077--1100, 2022.

\bibitem{MauwTX16}
S. Mauw, R. Trujillo{-}Rasua, and B. Xuan.
\newblock Counteracting active attacks in social network graphs.
\newblock In Silvio Ranise and Vipin Swarup, editors, {\em Data and
  Applications Security and Privacy {XXX} - 30th Annual {IFIP} {WG} 11.3
  Conference, DBSec 2016, Trento, Italy, July 18-20, 2016. Proceedings}, volume
  9766 of {\em Lecture Notes in Computer Science}, pages 233--248. Springer,
  2016.

\bibitem{kanonymity}
P. Samarati.
\newblock Protecting respondents identities in microdata release.
\newblock {\em IEEE Transactions on Knowledge and Data Engineering},
  13(6):1010--1027, 2001.

\bibitem{simanjuntak2019multiset}
R. Simanjuntak, P. Siagian, and T. Vetrik.
\newblock The multiset dimension of graphs, arXiv:1711.00225 [math.CO], 2019.

\bibitem{Slater1975}
P.~J. Slater.
\newblock Leaves of trees.
\newblock {\em Congressus Numerantium}, 14:549--559, 1975.

\bibitem{Tang2021}
J.-H. Tang, T. Noreen, M. Salman, M. Rehman, and J.-B. Liu.
\newblock $(k,\ell)$-anonymity in wheel-related social graphs measured on the
  base of $k$-metric antidimension.
\newblock {\em Journal of Mathematics}, 2021:1--13, 09 2021.

\bibitem{tillquist2023}
R.~C. Tillquist, R.~M. Frongillo, and M.~E. Lladser.
\newblock Getting the lay of the land in discrete space: A survey of metric
  dimension and its applications.
\newblock {\em SIAM Review}, 65(4):919--962, 2023.

\bibitem{TRUJILLORASUA2016403}
R. Trujillo-Rasua and I. {G. Yero}.
\newblock $k$-metric antidimension: A privacy measure for social graphs.
\newblock {\em Information Sciences}, 328:403--417, 2016.

\bibitem{Trujillo2016}
R. Trujillo-Rasua and I.~G. Yero.
\newblock {Characterizing 1-metric antidimensional trees and unicyclic graphs}.
\newblock {\em The Computer Journal}, 59(8):1264--1273, 08 2016.

\bibitem{Xpress}
Xpress.
\newblock {\em \em Fico$^{\tiny{\mbox{\textregistered}}}$ xpress solver}.

\bibitem{Zhang2017}
C. Zhang and Y. Gao.
\newblock On the complexity of $k$-metric antidimension problem and the size of
  $k$-antiresolving sets in random graphs.
\newblock In Yixin Cao and Jianer Chen, editors, {\em Computing and
  Combinatorics}, pages 555--567, Cham, 2017. Springer International
  Publishing.

\end{thebibliography}





\end{document}